\documentclass[a4paper]{article}

\usepackage{geometry}
\usepackage{amssymb, amsmath, latexsym, mathrsfs, verbatim, calc}
\usepackage{color}
\usepackage{xcolor}
\usepackage{indentfirst}\label{mmvPC.key}
\usepackage{amsthm}
\usepackage[round]{natbib}
\usepackage[bookmarks=true, bookmarksnumbered=true , bookmarksopen=true, bookmarksopenlevel = 2, colorlinks, linkcolor=cyan, anchorcolor=cyan, citecolor=cyan]{hyperref}   	
\usepackage{color}
\usepackage{graphicx}
\usepackage{subfigure}
\usepackage[titletoc,title]{appendix}
\usepackage{cases}

\newtheorem{assumption}{Assumption}[section]
\newtheorem{definition}{Definition}[section]
\newtheorem{proposition}{Proposition}[section]
\newtheorem*{problem}{Problem}
\newtheorem{lemma}{Lemma}[section]
\newtheorem{example}{Example}[section]
\newtheorem{experiment}{Experiment}[section]
\newtheorem{theorem}{Theorem}[section]
\newtheorem{corollary}{Corollary}[section]
\newtheorem{remark}{Remark}[section]
\newenvironment{proofth3.1}{{\noindent\bf Proof of Theorem \ref{mmvPC.theorem.Qcharacterization}}\quad}{\hfill $\square$\par}
\newenvironment{proofth4.1}{{\noindent\bf Proof of Theorem \ref{mmvPC.theorem.valuefunction}}\quad}{\hfill $\square$\par}
\newenvironment{proofth4.2}{{\noindent\bf Proof of Theorem \ref{mmvPC.theorem.verification}}\quad}{\hfill $\square$\par}
\newenvironment{proofth4.3}{{\noindent\bf Proof of Theorem \ref{mmvPC.theorem.efficientfrontier}}\quad}{\hfill $\square$\par}

\begin{document}
\title{\bf Optimal investment and reinsurance policies for the Cram{\'e}r-Lundberg risk model under monotone mean-variance preference}


\author{
	\renewcommand{\thefootnote}{\arabic{footnote}}
	Bohan Li\footnotemark[1]
	\and
	\renewcommand{\thefootnote}{\arabic{footnote}}
	Junyi Guo\footnotemark[2] $^*$
	\and
	\renewcommand{\thefootnote}{\arabic{footnote}}
	Linlin Tian\footnotemark[3]
}
\date{}

\renewcommand{\thefootnote}{\fnsymbol{footnote}}
\footnotetext{Bohan Li}
\footnotetext{bohanli@cuhk.edu.hk}
\footnotetext{}
\footnotetext{Junyi Guo}
\footnotetext{jyguo@nankai.edu.cn}
\footnotetext{*corresponding author}
\footnotetext{}
\footnotetext{Linlin Tian}
\footnotetext{linlin.tian@dhu.edu.cn}
\footnotetext{}

\renewcommand{\thefootnote}{\arabic{footnote}}
\footnotetext[1]{Department of Statistics, Chinese University of Hong Kong, Shatin, N.T., Hong Kong SAR.}
\footnotetext[2]{School of Mathematical Sciences, Nankai University, Tianjin 300071, People's Republic of China.}
\footnotetext[3]{College of Science, Donghua University, Shanghai 201620, People's Republic of China.}

\maketitle

\begin{abstract}
\noindent In this paper, an optimization problem for the monotone mean-variance(MMV) criterion is considered in the perspective of the insurance company. The MMV criterion is an amended version of the classical mean-variance(MV) criterion which guarantees the monotonicity of the utility function. With this criterion we study the optimal investment and reinsurance problem which is formulated as a zero-sum game between the insurance company and an imaginary player. We apply the dynamic programming principle to obtain the corresponding Hamilton-Jacobi-Bellman-Isaacs(HJBI) equation. As the main conclusion of this paper, by solving the HJBI equation explicitly, the closed forms of the optimal strategy and the value function are obtained. Moreover, the MMV efficient frontier is also provided. At the end of the paper, a numerical example is presented.
\end{abstract}
\vfill





\noindent


\section{Introduction}

\par In the past two decades, the mean-variance(MV) optimization problem has attracted considerable attention in the financial mathematics community, especially for the terminal wealth of an investor. The MV optimization problem, first proposed by \cite{markowitz1952portfolio}, is a multi-objective optimization problem:
\begin{equation}
\label{mmvPC.intro.multi.mv}
\text{Maximize} \bigg(-\mathbb{V}ar^PX^a,\mathbb{E}^PX^a\bigg), \quad s.t. \quad a \in \mathcal{A},
\end{equation}
where $\mathbb{V}ar^PX^a$ and $\mathbb{E}^PX^a$ are the variance and mean of the uncertain prospect $X^a$ controlled by the strategy $a$ under the probability measure $P$, and $\mathcal{A}$ is the set of admissible strategy. The goal of this problem is to find the maxima in the sense of Pareto optimality which together form the so-called efficient frontier in \cite{markowitz1952portfolio}. The strategies for reaching these maxima are the efficient strategies. Basically, the Pareto optimality means there is no other strategy which can make either $-\mathbb{V}ar^PX^a$ or $\mathbb{E}^PX^a$ better off without the other one being worse off.
\par In 2000, \cite{li2000optimal} extended the MV criterion to the multi-period setting using the idea of embedding the problem in a tractable auxiliary problem. From then on, the MV criterion became a popular objective in the field of multi-period optimization. For example, \cite{zhou2000continuous} studied the MV optimization problem in the continuous-time market. They introduced the stochastic linear-quadratic approach and obtained the efficient strategies as well as the efficient frontier by solving a stochastic Riccati equation. 
\par An application of the MV problem to the decision making of an insurance company is introduced by \cite{bauerle2005benchmark}. In this paper, the following criterion was used
\begin{equation*}
	\text{Minimize} \quad \mathbb{E}^P[X^a-\xi]^2,  \quad s.t. \quad a \in \mathcal{A}.
\end{equation*}
This criterion aims to make the insurance company's reserve at the terminal time remain close to a certain preassigned benchmark $\xi$. If we let $\xi$ equal $\mathbb{E}^PX^a$, the criterion with the benchmark becomes:
\begin{equation}
\label{mmvPC.intro.bench.mv}
\begin{cases}
&\text{Minimize} \quad \mathbb{E}^P[X^a-\xi]^2, \\
&\mathbb{E}^PX^a= \xi \quad s.t. \quad a \in \mathcal{A}.
\end{cases}
\end{equation}
Since the MV problem using criterion \eqref{mmvPC.intro.bench.mv} is proved to be equivalent to the multi-objective one \eqref{mmvPC.intro.multi.mv}, \cite{li2002dynamic} made use of criterion \eqref{mmvPC.intro.bench.mv} to solve the continuous time MV problem. Criterion \eqref{mmvPC.intro.bench.mv} simplifies the multi-objective MV problem to a state constrained optimization problem which allows us to apply the dynamic programming approach (studying the Hamilton-Jacobi-Bellman(HJB) equation). In particular, in the paper \cite{li2002dynamic}, the short-selling of stocks is prohibited which means the problem has no classical solution. Fortunately, they coped with this difficulty by finding the viscosity solution. If we let $\xi$ vary in an appropriate interval, the set of points $(\mathbb{E}^P[X^a-\xi]^2,\xi)$ is exactly the efficient frontier. In addition, \cite{bai2008optimal} investigated the no-shorting constrained MV problem where there are multiple risky assets. Also many other interesting aspects of MV problem are studied by some recent papers. \cite{bielecki2005continuous} considered the MV problem with bankruptcy prohibition, i.e., the state cannot be negative. To solve the problem, they decomposed it into two sub-problems and solved them separately. \cite{chen2013optimal} studied the problem via the maximum principle approach for a regime-switching market. Also by maximum principle approach, \cite{shen2014optimal} considered the problem when there is delayed capital inflow/outflow for the wealth process. \cite{shen2014mean} and \cite{sun2018optimal} studied the MV problem where the price processes of the stocks follow the constant elasticity of variance (CEV) model and the Heston stochastic volatility model, respectively. They obtained the explicit solutions by the Backwards Stochastic Differential Equation (BSDE) approach. {\color{magenta} There is also much literature considering the time-inconsistency of the MV problem, see \cite{landriault2018equilibrium}, \cite{ni2019equilibrium}, \cite{chen2021optimal}. Some other researchers incorporate the mean field type control theory to discuss the MV problem, see \cite{mei2020closed} and \cite{bensoussan2022dynamic}.}

\par By the method of Lagrange multipliers, criterion \eqref{mmvPC.intro.bench.mv} can be formulated as an unconstrained utility optimization problem given by
\begin{equation*}
\text{Maximize} \quad \mathsf{U}_{\xi}(X^a) =  - \mathbb{E}^P[X^a-\xi]^2 + l(\mathbb{E}^PX^a - \xi),
\end{equation*}
where $l \in \mathbb{R}$ and $\mathsf{U}_{\xi}$ is a utility function of the uncertain prospect $X^a$.

\par Here we introduce a strategically equivalent utility function called the penalty MV utility function (or MV preference in economics literature) as follows
\begin{equation}
	\label{mmvPC.intro.pref.mv}
	U_{\theta}(X^a) = \mathbb{E}^PX^a - \frac{\theta}{2} \mathbb{V}ar^PX^a,
\end{equation}
where $\theta$ is an index which measures the investor's aversion to risk. Let $\theta = 0$ then the utility becomes
\begin{equation}
	\label{mmvPC.intro.pref.mv.theta0}
	U_{0}(X^a) = \mathbb{E}^PX^a,
\end{equation}
which means that the investor wants to maximize the expected return regardless of risk. For large $\theta$, the original problem is approximately equivalent to maximizing
\begin{equation}
	\label{mmvPC.intro.pref.mv.theta8}
	U_{\infty}(X^a) = -\mathbb{V}ar^PX^a,
\end{equation}
which means the investor is absolutely risk averse and it wants to minimize the risk regardless of the expected return.

\par Although the MV preference \eqref{mmvPC.intro.pref.mv} is successful in the field of finance and economics due to the intuitive meaning and tractability, it has an inevitable and crucial weakness in that it may fail to be monotone. This is irrational in the economic context, since it may happen that $X < Y$ but $U_{\theta}(X) > U_{\theta}(Y)$. In practice, if an investor uses the MV preference as its objective to optimize the wealth, it may happen that the optimal strategy doesn't actually optimize the absolute amount of the underlying wealth. To overcome this shortcoming, \cite{maccheroni2009portfolio} introduced an amended version of the MV preference called the monotone mean-variance(MMV) preference (utility function) which is exactly monotonic with respect to the uncertain prospect $X$. Specifically, the MMV preference is given by
\begin{equation}
	\label{mmvPC.intro.pref.mmv}
	V_{\theta}(X^a) = \min_Q \bigg\{ \mathbb{E}^QX^a+\frac{1}{2 \theta} C(Q||P) \bigg\},
\end{equation}
where $Q$ ranges over all the absolutely continuous probability measures whose densities are square-integrable with respect to $P$, and $C(Q||P)$ is the relative Gini concentration index. The parameter $\theta$ in \eqref{mmvPC.intro.pref.mmv} is the same as that in \eqref{mmvPC.intro.pref.mv}.
For a similar objective function, readers are referred to \cite{hansen2006robust} and \cite{Elliott2009Portfolio}.

\par It is proved that the MMV preference $V_{\theta}$ in \eqref{mmvPC.intro.pref.mmv} is the minimal monotone modification of the MV preference $U_{\theta}$. Moreover, $V_{\theta}$ agrees with $U_{\theta}$ in the area where $U_{\theta}$ is monotonic. It follows from Lemma 2.1 of \cite{maccheroni2009portfolio} that the domain of monotonicity of $U_{\theta}$ consists of $X$ that is dominated by $\mathbb{E}^P[X] + \frac{1}{\theta}$. This is too constricted when $\theta$ is not small enough, and most of the meaningful variables in finance and economics fail to meet this condition. If we allow $\theta$ to vary in $(0,\infty)$ so as to get the efficient frontier, the domain of monotonicity becomes $\{X: X \leq \mathbb{E}^P[X]\}$, which is especially constricted. In this paper, the uncertain prospect $X$ is the terminal wealth of the insurance company who invests in the capital market and purchases reinsurance. In this case, $X$ may not satisfy this condition under some strategies. For this reason, we use the MMV preference instead of the classical one as the criterion to optimize the insurance company's terminal wealth. We obtain in this paper the optimal strategies $a^{\theta}$ for different parameters $\theta$. Denote $(-\mathbb{V}ar^PX^{a^{\theta}},\mathbb{E}^PX^{a^{\theta}})_{\theta \in \Theta}$ as the efficient frontier of the optimal MMV problem. A very interesting result is that the efficient frontier of MMV optimization problem is the same as that of the classical MV problem.

\par In the literature, there are few articles about the optimization problem of the MMV criterion. \cite{Trybu2014Continuous} first investigate it in the framework of continuous time financial market and it is assumed that the interest rate is a random process driven by a Brownian motion. In \cite{trybula2019continuous-time}, a continuous time portfolio choice problem is studied where the authors allow the coefficients of the stock prices process to be stochastic. They assume that $Q$ is an equivalent probability measure with respect to $P$. The optimal portfolio and the value function are obtained when the coefficients are specified. For a large class of portfolio choice problem, \cite{strub2020note} further prove that, when the risk assets are continuous semimartingale, the optimal portfolio and the value function of the classical MV preference and the MMV preference coincide. 
{\color{red}
\par We list the contributions of this paper as follows:
\begin{itemize}
	\item Although there is much literature in recent years for the MMV optimization problem, there is no paper studying the present problem involving jump processes. In this paper, the stock price process follows the Black-Scholes model, and the insurance surplus follows the Cram{\'e}r-Lundberg model where the aggregate claims process is a jump process.
	\item In contrast to \cite{Trybu2014Continuous} and \cite{trybula2019continuous-time}, where $Q$ is assumed to be chosen in the set of equivalent probability measure of $P$, in this paper, we allow $Q$ to be absolutely continuous with respect to $P$, which is more consistent with the original definition of the MMV preference in \cite{maccheroni2009portfolio}. It is difficult to consider the absolutely continuous probabilities. Since the Radon-Nikodym derivative $\frac{dQ}{dP}$ is not necessarily positive, letting $\mathcal{F}_t$ be the natural filtration, we can not simply write $\frac{dQ}{dP}\big|_{\mathcal{F}_t}$ as an exponential martingale( which is always positive). However, we managed to resolve this problem for the continuous state process case in our former work \cite{li2021optimal}. And in the present jump-involving framework, the problem can be more difficult to tackle with. In the later context, we can show that the minimum point is reached at an equivalent probability measure $Q^*$.  
	\item It is the first time that the MMV preference is incorporated into the optimal investment and reinsurance problem for the Cram{\'e}r-Lundberg risk model. The objective is to find the optimal portfolio and optimal retention level to maximize the MMV criterion of the terminal wealth of the insurance company. The explicit solutions are obtained.
\end{itemize}
}

\par The paper is organized as follows. In Section 2, we present the definition of the MMV preference proposed by \cite{maccheroni2009portfolio} and then introduce the insurance risk model in the financial market. In Section 3, we give an auxiliary two-player zero-sum stochastic differential game(SDG). This SDG can be solved by applying the dynamic programming principle and solving the corresponding Hamilton-Jacobi-Bellman-Isaacs(HJBI) equation. The original problem is then resolved consequently. Section 4 provides the value function and the optimal strategy explicitly. The efficient frontier for MMV is presented in Section 5. In Section 6, an example is presented to show the monotonicity of the MMV criterion. Besides, a simulation of the optimal strategies in the financial and insurance market is conducted.

\section{Model setup}

\subsection{Monotone mean-variance criterion}

\par The following MMV preference introduced by \cite{maccheroni2009portfolio} is used as the objective function,
\begin{equation}
\label{mmvPC.obj.0} V_{\theta}(X) = \min_{Q \in \Delta^2(P)} \bigg\{ \mathbb{E}^Q[X]+\frac{1}{2 \theta} C(Q||P) \bigg\}, \quad \forall X \in \mathbb{L}^2(P),
\end{equation}
where
\begin{equation*}
\Delta^2(P) = \bigg\{Q \ll P: Q( \Omega ) = 1, \mathbb{E}^P \bigg[\left(\frac{dQ}{dP}\right)^2\bigg] < \infty \bigg\},
\end{equation*}
and $C(Q||P)$ is defined as
\begin{equation*}
C(Q||P) = \mathbb{E}^P \bigg[\left(\frac{dQ}{dP}\right)^2 \bigg] - 1,
\end{equation*}
which is called the relative Gini concentration index (or $\chi ^2$-distance) and enjoys properties similar to those of the relative entropy (see \cite{liese2007convex}). Here $\theta$ is an index measuring the risk aversion of the insurance company.

\begin{problem}{\bf (MMV$_{\theta}$)}
	\begin{equation*}
	\text{\rm Maximize} \quad I_{\theta}(a) \quad \text{\rm over} \quad \mathcal{A},
	\end{equation*}
	where
	\begin{equation}
	\label{mmvPC.obj.1} I_{\theta}(a): = V_{\theta}(X^a) = \inf_{Q \in \Delta^2(P)} \left \{ \mathbb{E}^Q[X^a]+\frac{1}{2 \theta} C(Q||P) \right \},
	\end{equation}
	$a$ is the control process of $X$ and $\mathcal{A}$ is the admissible set.
\end{problem}

\par This is a max-min optimization problem, and is naturally relative to a zero-sum game as follows.

\begin{problem}{\bf (G)} \label{mmvPC.problem.G}
	\par Let
	\begin{equation*}
	J_{\theta}(a,Q): =  \left \{ \mathbb{E}^Q[X^a]+\frac{1}{2 \theta} C(Q||P) \right \}  .
	\end{equation*}
	The player one wants to maximize $J_{\theta}(a,Q)$ with its strategy $a$ over $\mathcal{A}[0,T]$ and the player two wants to maximize $-J_{\theta}(a,Q)$ with its strategy $Q$ over $\Delta^2(P)$.
\end{problem}

\par To explain Problem {\bf (G)} in the sense of economics, we regard it as a game between the market participant (in our paper, the insurance company) and the market (an imaginary player). $a$ is the strategy of the insurance company, $Q$ is the strategy of the market. Except for the insurance company, there are many other market participants (or competitors of the insurance company) in the market. Once any new opportunity for profits appears, it will be exploited by the competitors immediately, which may cause the market probability measure to deviate from the real-world probability measure $P$ to some new absolutely continuous probability measure $Q$. For this reason, $C(Q||P)$ describes the deviations of the probability measure and the cost of market's strategy (competitors' strategy). This explains that the market is to minimize $\mathbb{E}^Q[X]$ with the cost of $C(Q||P)$.

\par The parameter $\theta$ here measures the difficulty for the market to take a strategy. If $\theta = 0$, the market is extremely steady such that nobody can change the probability measure of the market, in other words, the competitors cannot grab opportunities but leave them there. Hence $Q \equiv P$ and Problem {\bf (G)} degenerates to \eqref{mmvPC.intro.pref.mv.theta0}. If $\theta = \infty$, there is no cost to change the probability measure of the market. In this case, whatever the insurance company's strategy is, only a riskless return can be earned, since all the opportunities are exploited immediately by other competitors once it appears. Therefore, Problem {\bf (G)} degenerates to \eqref{mmvPC.intro.pref.mv.theta8}.

\par The following statement illustrates the equivalence between Problem {\bf (MMV$_{\theta}$)} and Problem {\bf (G)}. We set
\begin{equation*}
I^{\sharp}_{\theta}(Q): =  \sup_{a \in \mathcal{A}} J_{\theta}(a,Q), \quad \Phi^{\sharp}_{\theta}: =  \inf_{Q \in \Delta^2(P)} I^{\sharp}_{\theta}(Q), \quad \Phi_{\theta} = \sup_{a \in \mathcal{A}}I_{\theta}(a).
\end{equation*}

\begin{definition} \label{mmvPC.definition.1}
	\par If there exists $a^* \in \mathcal{A}$ and a probability measure $Q^* \in \Delta^2(P)$ such that
	\begin{align*}
	& J_{\theta}(a^*,Q^*) = \sup_{a \in \mathcal{A}} J_{\theta}(a,Q^*), \\
	& J_{\theta}(a^*,Q^*) = \inf_{Q \in \Delta^2(P)} J_{\theta}(a^*,Q),
	\end{align*}
	then we call the pair $(a^*,Q^*)$ a Nash equilibrium( non-cooperative equilibrium) for Problem {\bf (G)}.
\end{definition}

\begin{definition} \label{mmvPC.definition.2}
	\par We call $a^* \in \mathcal{A}$ an optimal strategy for Problem {\bf (MMV$_{\theta}$)} if $I_{\theta}(a^*) = \Phi_{\theta}$.
\end{definition}

\begin{lemma} \label{mmvPC.lemma.1}
	\par Given some $a^* \in \mathcal{A}$ and $Q^* \in \Delta^2(P)$, the following statements are equivalent:
	\par \noindent (a) $(a^*,Q^*)$ is a Nash equilibrium of Problem {\bf (G)};
	\par \noindent (b) $\forall (a,Q) \in \mathcal{A} \times \Delta^2(P)$, $J_{\theta}(a,Q^*) \leq J_{\theta}(a^*,Q^*) \leq J_{\theta}(a^*,Q)$;
	\par \noindent (c) $I^{\sharp}_{\theta}(Q^*) = \Phi^{\sharp}_{\theta} = \Phi_{\theta} = I_{\theta}(a^*)$.
\end{lemma}
\begin{proof}
	\par The proof is the same as that of Proposition 8.1 in \cite{aubin2013optima}.
\end{proof}

\par Lemma \ref{mmvPC.lemma.1} tells us that the optimal strategy of Problem {\bf (MMV$_{\theta}$)} actually lies in the Nash equilibrium of Problem {\bf (G)}.

\subsection{Financial market and insurance market} \label{mmvPC.section.market}

\par Let $P$ be the real world probability measure. The following standard Black-Scholes type financial market is considered.
\begin{align}
\label{mmvPC.stocks.origin}
\begin{cases}
B(t) =& B(0) + \int_0^t r(s)B(s)ds, \\
S(t) =& S(0) + \int_0^t \mu(s)S(s) ds+   \int_0^t \sigma(s) S(s) dW(s), \quad t \in [0,T],
\end{cases}
\end{align}
where $B(t)$ is the riskless asset, $S(t)$ is the price of the risky stock and $W(t)$ is a standard Brownian motion. Denote $\mathbb{F}^W:=\{ \mathcal{F}_t^W: t \leq T \}$, where $\mathcal{F}_t^W$ is the completion of $\sigma(W(s): s \leq t)$ under $P$. 

{\color{red}
\begin{assumption} \label{mmvPC.assumption.1}
	\par Assume that $r$, $\mu$ and $\sigma$ are all bounded positive deterministic functions such that $\mu(t) > r(t)$ for all $t \in [0,T]$ and $\sigma$ satisfies the non-degenerate condition: there exists $\underline{\sigma} >0$ such that $\sigma(t) > \underline{\sigma}$ for all $t \in [0,T]$. 
\end{assumption}
}
\par On the other hand, the following controlled Cram{\'e}r-Lundberg model is considered as the insurance surplus process
\begin{align*}
dR(t) = (1+\kappa)\mu_0(t) dt - (1+ \kappa_r)(1-u(t))\mu_0(t) dt - u(t) dC(t),
\end{align*}
where $C(t)$ is the claim process, $\mu_0(t) := \frac{d\mathbb{E}^PC(t)}{dt}$, $\kappa$ and $\kappa_r$ are the safety loadings of the insurer and the reinsurer respectively, $u$ is the control process representing the retention level of the insurance company. Generally, $0 < \kappa \leq \kappa_r$. Here $C(t) = \int_0^t \int_{K} z N(ds,dz)$ with a Poisson random measure $N(ds,dz)$ counts the number of claims within the time interval $[s,s+ds)$ whose sizes lie in the interval $[z,z +dz)$. $K$ is the support of $N(ds,dz)$. Let $v(dz)$ be the L\'{e}vy measure of the random measure $N(ds,dz)$ and $\widetilde{N}(ds,dz) = N(ds,dz) - v(dz)ds$ be the compensated Poisson random measure. Then $\mu_0(t) \equiv \mu_0 = \int_{K} z v(dz) < +\infty$. The surplus process becomes
\begin{align}
\label{mmvPC.surplus}
dR(t) =\bigg[\kappa_r u(t) + (\kappa - \kappa_r)\bigg]\mu_0dt - u(t) \int_{K}z \widetilde{N}(dt,dz).
\end{align}
Denote by $\mathbb{F}^N:=\{ \mathcal{F}_t^N : t \leq T \}$, where $\mathcal{F}_t^N$ is the completion of the $\sigma$-field generated by the Poisson measure $N(ds,dz)$ under $P$. By Proposition 2.7.7 of \cite{karatzas2012brownian}, $\mathbb{F}^N$ is a right-continuous filtration. Generally, the insurance market is independent of the capital market, i.e., the filtration $\mathbb{F}^W$ is independent of $\mathbb{F}^N$. Denote $\mathbb{F} := \{ \mathcal{F}_t: t \leq T \} := \{ \mathcal{F}_t^W \bigvee \mathcal{F}_t^N: t \leq T \}$ as the combined filtration of $\mathbb{F}^W$ and $\mathbb{F}^N$, where $\mathcal{F}_t^W \bigvee \mathcal{F}_t^N$ is the smallest $\sigma$-field containing both $\mathcal{F}_t^W$ and $\mathcal{F}_t^N$.

\par The insurance company's wealth process $X(t)$ can be formulated as the following controlled process
\begin{align}
	dX(t)
	\nonumber =& \pi(t)\frac{dS(t)}{S(t)}+(X(t)-\pi(t))\frac{dB(t)}{B(t)}+dR(t)\\
	\nonumber =&  \bigg[r(t)X(t)+\pi(t) (\mu(t) - r(t))  + \mu_0 \kappa_r u(t)+\mu_0(\kappa - \kappa_r)\bigg]dt \\
	\label{mmvPC.wealth} &- u(t) \int_{K} z \widetilde{N}(dt,dz) +\pi(t) \sigma(t) dW(t), \quad t \in [0,T],
\end{align}
with initial value $X(0) = x_0$. For notational convenience, let $a := (\pi,u)$ and
\begin{equation*}
	\rho(t) := \frac{\mu_0^2 \kappa_r^2}{\sigma_0^2} + \frac{(\mu(t) - r(t))^2}{\sigma(t)^2},
\end{equation*}
where  $\sigma_0^2 := \int_{K} z^2 v(dz) < +\infty$. Since $\frac{\mu_0^2 \kappa_r^2}{\sigma_0^2}$ is the cost of risk-transfer in insurance market, and $\frac{(\mu(t) - r(t))^2}{\sigma(t)^2}$ is the capital market risk premium, $\rho(t)$ measures how much benefit the insurance company can derive from undertaking the risk.


\section{Auxiliary SDE}

\par To solve Problem {\bf (G)} for the prospect $X(T)$, i.e., the terminal wealth of the insurance company, it is intuitive to characterize $Q$ by some process which is more tractable than a probability measure. If the probability measure $Q$ is constrained to a subset of $\Delta^2(P)$:
\begin{equation*}
\widetilde{\Delta}^2(P) = \bigg\{Q \sim P: Q( \Omega ) = 1, \mathbb{E}^P \bigg[\left(\frac{dQ}{dP}\right)^2\bigg] < \infty \bigg\}
\end{equation*}
($Q \sim P$ means $Q$ is equivalent to $P$), we can take a similar procedure as that of \cite{mataramvura2008risk} and \cite{Elliott2009Portfolio} to characterize $Q$ by the following SDE
\begin{equation}
\label{mmvPC.Q.representation}
Y(t) = 1  + \int_0^t Y(s) p(s)dW(s) + \int_0^t \int_{K} Y(s-) q(s,z) \widetilde{N}(ds,dz), \quad \forall t \in [0,T],
\end{equation}
where $q(s,z)$ is a $\mathcal{F}_t$-predictable process and $p(s)$ is a $\mathcal{F}_t$-adapted process. Here $\frac{dQ}{dP}\big|_{\mathcal{F}_t}$ is the unique solution of SDE \eqref{mmvPC.Q.representation}(see \cite{mataramvura2008risk} and \cite{Elliott2009Portfolio} for the cases without jumps). Specially, we can regard $q(s,z)$ and $p(s)$ as control processes which have a one-to-one correspondence with $Q \in \widetilde{\Delta}^2(P)$. 

\par In this paper, we consider the absolutely continuous probability measures $\Delta^2(P) = \big\{Q \ll P: Q( \Omega ) = 1, E\left[\left(\frac{dQ}{dP}\right)^2\right] < \infty \big\}$. In this case, $\frac{dQ}{dP}\big|_{\mathcal{F}_t}$ may equal zero at some $t \in [0,T]$. To overcome this difficulty, we take a similar procedure as shown in \cite{li2021optimal} in which a continuous case is considered. By using the tools of discontinuous martingale analysis in \cite{kabanov1979absolute}, \cite{liptser2012theory} and \cite{hansen2006robust}, it can be proved that the martingale representation \eqref{mmvPC.Q.representation} still holds for all $Q \in \Delta^2(P)$(Theorem \ref{mmvPC.theorem.Qcharacterization}). We can then still make use of the auxiliary Problem {\bf (P$_{sxy}$)} to help solve the original problem. The proof is in the appendices. The original problem can be replaced by the following auxiliary problem
 
\begin{numcases}{}
\notag dX(t) =  \bigg[r(t)X(t)+\pi(t) (\mu(t) - r(t))  + \mu_0 \kappa_r u(t)+\mu_0(\kappa - \kappa_r)\bigg]dt \\
\label{mmvPC.dynamic.x} \qquad \qquad - u(t) \int_{K} z \widetilde{N}(dt,dz) +\pi(t) \sigma(t) dW(t), \\
\label{mmvPC.dynamic.y} dY(t) = Y(t) p(t)dW(t) + \int_{K} Y(t-) q(t,z) \widetilde{N}(dt,dz) , \quad t \in [s,T],
\end{numcases} 
with $\int_{K} q(t,z) N(\{t\},dz) \geq -1$, $t \in [0,T]$ and initial values $X(s) = x > 0$, $Y(s) = y > 0$.
 
\begin{problem}{\bf (P$_{sxy}$)} Let
	\begin{equation}
	\label{mmvPC.obj.3}
	J^{a,b}(s,x,y) = \mathbb{E}_{s,x,y}^P \bigg[X^{a}(T)Y^{b}(T)+\frac{1}{2\theta}(Y^{b}(T))^2 \bigg],
	\end{equation}
	where $\mathbb{E}_{s,x,y}^P[\cdot]$ represents $\mathbb{E}^P[\cdot|X(s) = x,Y(s) = y]$. The player one wants to maximize $J^{a,b}(s,x,y)$ with its strategy $a = (\pi,u)$ over $\mathcal{A}[s,T]$ defined below and the player two wants to maximize $-J^{a,b}(s,x,y)$ with its strategy $b = (p,q)$ over $\mathcal{B}[s,T]$ defined below.
\end{problem}

\begin{definition} \label{mmvPC.definition.4}
	\par {\bf (admissible)} The strategy $a(t) = (\pi(t),u(t))$ of the player one is admissible for Problem {\bf (P$_{sxy}$)}, if $u: [s,T] \to \mathbb{R}_+$ and $\pi: [s,T] \to \mathbb{R}$ are $\mathcal{F}_t$-predictable process and $\mathcal{F}_t$-adapted process, respectively, such that
	\begin{equation*}
	\mathbb{E}^P \int_s^T \pi(t)^2 + u(t)^2 dt < \infty,
	\end{equation*}
	moreover, we denote $\mathcal{A}[s,T]$ as the set of all admissible strategies $a(t) = (\pi(t),u(t))$.
	\par The strategy $b(t,z) = (p(t),q(t,z))$ of the player two is admissible for Problem {\bf (P$_{sxy}$)}, if for any fixed $z \in K$, $q(\cdot,z): [s,T] \to \mathbb{R}$ and $p: [s,T] \to \mathbb{R}$ are $\mathcal{F}_t$-predictable process and $\mathcal{F}_t$-adapted process, respectively, satisfying
	\begin{equation*}
	\Delta L(t) \equiv \int_{K} q(t,z) N(\{t\},dz) \geq -1, \quad t \in [s,T],
	\end{equation*}
	 such that SDE \eqref{mmvPC.dynamic.y} has a unique solution which is a {\em nonnegative} $\mathcal{F}_t$-adapted square-integrable $P$-martingale satisfying $\mathbb{E}^PY(t) = 1$ for $t \in [s,T]$. Moreover, we denote $\mathcal{B}[s,T]$ as the set of all admissible strategies $b(t,z) = (p(t),q(t,z))$.
\end{definition}
{\color{red}
In general, the SDE \eqref{mmvPC.dynamic.y} admits an exponential martingale solution (positive almost everywhere) if $p(t)$ and $q(t,z)$ satisfy some additional conditions such as the Novikov's condition. However, we do not include these additional conditions in Definition \ref{mmvPC.definition.4}. The SDE \eqref{mmvPC.dynamic.y} can also admit {\em nonnegative} solutions which is enough in our paper even if the Novikov's condition does not hold.
\begin{example}
	Let $o(t) = p(t,z) := 0$ and $q(t,z) :=\eta(t-)$ where $\eta(t) = \frac{1}{\tau} Z(t)^{-1}1_{\{t < \tau\}}$ and $Z(t)$ is a square-integrable marginale:
	\begin{align*}
	Z(t) := \begin{cases}
	&\frac{N_2(t) - t + \tau}{\tau}, \quad 0 \leq t < \tau, \\
	&\frac{N_2(\tau)}{\tau}, \quad \tau \leq t \leq T.
	\end{cases}
	\end{align*}
	Then the SDE \eqref{mmvPC.sde.y.representation} becomes
	\begin{align}
	\label{mmvPC.dynamic.y.example}
	Y(t) = 1 + \int_0^t Y(s-) q(s) d\widetilde{N}_2(s).
	\end{align}
	It is easy to check that \eqref{mmvPC.dynamic.y.example} admits a solution $Y(t) = Z(t)$(indistinguishable). Indeed, if $0 \leq t < \tau$, then
	\begin{align*}
		1 + \int_0^t Z(s-) q(s) d\widetilde{N}_2(s) =& 1 + \int_0^t Z(s-) \eta(s-) d\widetilde{N}_2(s) \\
		=& 1 + \int_0^t Z(s-)\frac{1}{\tau} Z(s-)^{-1}1_{\{s \leq \tau\}} d\widetilde{N}_2(s) \\
		=& 1 + \frac{N_2(t)-t}{\tau} = Z(t).
	\end{align*}
	Meanwhile, if $\tau \leq t$, then
	\begin{align*}
		1 + \int_0^t Z(s-) q(s) d\widetilde{N}(s) =& 1 + \int_0^{\tau} Z(s-) \eta(s-) d\widetilde{N}_2(s) + \int_{\tau}^t Z(s-) \eta(s-) d\widetilde{N}_2(s) \\
		=& 1 + \int_0^{\tau} Z(s-)\frac{1}{\tau} Z(s-)^{-1}1_{\{s \leq \tau\}} d\widetilde{N}_2(s) \\
		=& 1 + \frac{N_2(\tau)-\tau}{\tau} = \frac{N_2(\tau)}{\tau} = Z(t).
	\end{align*}
	Therefore, $1 + \int_0^t Z(s-) q(s) d\widetilde{N}_2(s) = Z(t)$ for any $t \in [0,T]$. By substituting $Z(t)$ into $q(t)=\eta(t-)= \frac{1}{\tau} Z(t-)^{-1}1_{\{t \leq \tau\}}$, we have
	\begin{align*}
	q(t) =
	\begin{cases}
	&\frac{1}{N_2(t-)-(t-)+\tau}, \quad 0 \leq t \leq \tau, \\
	&0, \quad \tau < t \leq T.
	\end{cases}
	\end{align*}
	We can notice that on the set of positive measure $\{\omega: N_2(\tau) = 0\}$,
	\begin{align*}
	\int_0^T\int_{\mathbb{R_+}}q(t,z)F_2(dz)dt =& \int_0^{\tau}\frac{1}{N_2(t)-t+\tau}dt \\
	=& \int_0^{\tau}\frac{1}{\tau-t}dt = \infty,
	\end{align*}
	which fails to satisfy the Novikov's condition.
\end{example}
}

\par The infinitesimal generator of $(X(t),Y(t))$ for any smooth function $\phi(t,x,y)$ is given by
\begin{align*}
	\mathcal{L}_t^{a,b} \phi(t,x,y) =& \frac{\partial \phi(t,x,y)}{\partial t}+\bigg[r(t)x + \pi (\mu(t)-r(t)) + \mu_0 \kappa_r u +\mu_0(\kappa - \kappa_r)\bigg]\frac{\partial \phi(t,x,y)}{\partial x} \\
	&+\frac{1}{2} \pi^2 \sigma(t)^2 \frac{\partial ^2 \phi(t,x,y)}{\partial x^2} +\frac{1}{2} y^2 p^2 \frac{\partial ^2 \phi(t,x,y)}{\partial y^2} + y \pi \sigma(t) p \frac{\partial ^2 \phi(t,x,y)}{\partial x \partial y} \\
	&+ \int_{K} \bigg\{ \phi(t,x-uz,y+yq(z)) - \phi(t,x,y) + u z \frac{\partial \phi(t,x,y)}{\partial x} - y q(z) \frac{\partial \phi(t,x,y)}{\partial y} \bigg\} v(dz).
\end{align*}

\par We only consider here that $a(t)$ and $b(t)$ are both Markov feedback control, that is,  $a(t) = a(t,X(t),Y(t))$ and $b(t,z) = b(t,z,X(t),Y(t))$.

\par It follows from \cite{yeung2006cooperative} and \cite{mataramvura2008risk} that the HJBI equation is given by
\begin{align}
	\label{mmvPC.hjbi.equation}
	\begin{cases}
	&\mathcal{L}_t^{\hat{a},\hat{b}}\phi(t,x,y) = 0, \qquad \forall (t,x,y) \in [0,T) \times \mathbb{R} \times \mathbb{R}_+, \\
	&\mathcal{L}_t^{\hat{a},b}\phi(t,x,y) \geq 0, \qquad \forall b \in \mathbb{R}^{2}, \qquad \forall (t,x,y) \in [0,T) \times \mathbb{R} \times \mathbb{R}_+, \\
	&\mathcal{L}_t^{a,\hat{b}}\phi(t,x,y) \leq 0, \qquad \forall a \in \mathbb{R} \times \mathbb{R}_+, \qquad \forall (t,x,y) \in [0,T) \times \mathbb{R} \times \mathbb{R}_+, \\
	&\phi(T,x,y) = xy + \frac{1}{2 \theta} y^2,
	\end{cases}
\end{align}
where the feedback pair $\big(\hat{a}(t,x,y), \hat{b}(t,z,x,y)\big)$ is called the saddle point.

\section{Optimal strategies}
\par In this section, we give the value function and the optimal strategy for Problem {\bf (MMV$_{\theta}$)}. The following lemma gives a candidate for the Nash equilibrium of Problem {\bf (P$_{sxy}$)}.

\begin{theorem} \label{mmvPC.theorem.valuefunction}
\par {\bf (solution of HJBI equation)} The following function $\phi(t,x,y) \in C^{1,2,2}([0,T] \times \mathbb{R} \times \mathbb{R}_+)$ is a solution of the HJBI equation \eqref{mmvPC.hjbi.equation}
\begin{align}
	\nonumber \phi(t,x,y) =& \exp \left(\int_t^T r(s) ds \right) xy + \frac{1}{2 \theta} \exp \left(\int_t^T \rho(s) ds \right) y^2\\
	\label{mmvPC.value.function} &+  \mu_0(\kappa - \kappa_r) y \int_t^T  \exp \left(\int_s^T r(v) dv \right)ds.
\end{align}
where $\rho(t) = \frac{\mu_0^2 \kappa_r^2}{\sigma_0^2} + \frac{(\mu(t)-r(t))^2}{\sigma(t)^2}$. The saddle point of the HJBI equation \eqref{mmvPC.hjbi.equation} is given by
\begin{align}
\label{mmvPC.saddle}
\begin{cases}
&\hat{a}(t,x,y) = (\hat{\pi}(t,x,y),\hat{u}(t,x,y)) = \frac{1}{\theta}  \bigg(\frac{\mu(t)-r(t)}{\sigma(t)^2}, \frac{\mu_0 \kappa_r }{ \sigma_0^2} \bigg) \times \exp \left( \int_t^T \rho(v)-r(v) dv \right)y, \\
&\hat{b}(t,z,x,y) = (\hat{p}(t,x,y),\hat{q}(t,z,x,y)) = \bigg(-\frac{\mu(t)-r(t)}{\sigma(t)}, \frac{\mu_0 \kappa_r z}{\sigma_0^2} \bigg).
\end{cases}
\end{align}
\end{theorem} 

\begin{proof}
	\par We study the following nonlinear equation
	\begin{align}
	\nonumber 0 =& \sup_{a \in \mathbb{R} \times \mathbb{R}_+}  \inf_{b \in \mathbb{R}^{2}} \mathcal{L}_t^{a,b} \phi(t,x,y)  \\
	\nonumber =& \sup_{a \in \mathbb{R} \times \mathbb{R}_+} \inf_{b \in \mathbb{R}^{2}} \{ \frac{\partial \phi(t,x,y)}{\partial t}+\bigg[r(t)x + \pi (\mu(t)-r(t)) + \mu_0 \kappa_r u +\mu_0(\kappa - \kappa_r)\bigg]\frac{\partial \phi(t,x,y)}{\partial x} \\
	\nonumber &+\frac{1}{2} \pi^2 \sigma(t)^2 \frac{\partial ^2 \phi(t,x,y)}{\partial x^2} +\frac{1}{2} y^2 p^2 \frac{\partial ^2 \phi(t,x,y)}{\partial y^2} + y \pi \sigma(t) p \frac{\partial ^2 \phi(t,x,y)}{\partial x \partial y} \\
	\label{mmvPC.hjbi.equation.compact} &+ \int_{K} \bigg\{ \phi(t,x-uz,y+yq(z)) - \phi(t,x,y) + u z \frac{\partial \phi(t,x,y)}{\partial x} - y q(z) \frac{\partial \phi(t,x,y)}{\partial y} \bigg\} v(dz) \},
	\end{align}
	with the terminal condition $\phi(T,x,y) = xy + \frac{1}{2 \theta} y^2$. Inspired by the terminal condition of \eqref{mmvPC.hjbi.equation.compact}, we try a solution of the following form
	\begin{equation}
	\label{mmvPC.guess}
	\phi(t,x,y) = \Lambda(t)xy + \Theta(t)y^2 + \Psi(t)y,
	\end{equation}
	where $\Lambda(T)=1$, $\Theta(T) = \frac{1}{2\theta}$ and $\Psi(T) = 0$. Therefore,
	\begin{align}
	\nonumber &\frac{\partial \phi}{\partial t} = \Lambda'(t)xy+\Theta'(t)y^2+\Psi'(t)y, \quad \frac{\partial \phi}{\partial x} = \Lambda(t)y, \quad \frac{\partial ^2 \phi}{\partial x^2} = 0, \\
	\label{mmvPC.phi.differential} &\frac{\partial \phi}{\partial y} = \Lambda(t)x + 2 \Theta(t)y + \Psi(t), \quad \frac{\partial ^2 \phi}{\partial y^2} = 2\Theta(t), \quad \frac{\partial ^2 \phi}{\partial x \partial y} = \Lambda(t).
	\end{align}
	
	\par Substituting \eqref{mmvPC.phi.differential} into \eqref{mmvPC.hjbi.equation.compact} gives
	\begin{align}
	\nonumber 0 =& \sup_{a \in \mathbb{R} \times \mathbb{R}_+}  \{ \Lambda'(t)xy+\Theta'(t)y^2+\Psi'(t)y \\
	\nonumber &+\bigg[r(t)x + \pi (\mu(t)-r(t)) + \mu_0 \kappa_r u +\mu_0(\kappa - \kappa_r)\bigg]\Lambda(t)y \\
	\nonumber & + \inf_{p \in \mathbb{R}} \bigg[y^2 p^2 \Theta(t) + y \pi \sigma(t) p \Lambda(t)\bigg] \\
	\label{mmvPC.hjbi.equation.compact.guess} &+ \int_{K} \inf_{q \in \mathbb{R}} \bigg[ -\Lambda(t)uzyq(z) + \Theta(t)y^2q(z)^2 \bigg] v(dz) \}.
	\end{align}

	\par To study the equation \eqref{mmvPC.hjbi.equation.compact.guess}, first let $a \in \mathbb{R} \times \mathbb{R}_+$ be arbitrary and solve
	\begin{equation*}
	\begin{cases}
	\inf_{p \in \mathbb{R}} \bigg[y^2 p^2 \Theta(t) + y \pi \sigma(t) p \Lambda(t)\bigg],\\
	\inf_{q \in \mathbb{R}} \bigg[ -\Lambda(t)uzyq(z) + \Theta(t)y^2q(z)^2 \bigg],
	\end{cases}
	\end{equation*}
	pointwisely over the jump size $z \in K$. If $\Theta(t) > 0$, the minimum is attained at
	\begin{equation}
	\label{mmvPC.q.minimizer}
	\begin{cases}
	\hat{p}(t,x,y) = - \frac{\Lambda(t)}{2 y \Theta(t)} \pi \sigma(t), \\
	\hat{q}(t,z,x,y) = \frac{\Lambda(t)}{2 y \Theta(t)}u z.
	\end{cases}
	\end{equation}
	
	\par Substituting \eqref{mmvPC.q.minimizer} into \eqref{mmvPC.hjbi.equation.compact.guess} gives
	\begin{align}
	\nonumber 0 =&  \Lambda'(t)xy+\Theta'(t)y^2+\Psi'(t)y + [r(t)x + \mu_0(\kappa - \kappa_r)]\Lambda(t)y \\
	\nonumber &+ \sup_{\pi \in \mathbb{R}} \bigg[\pi (\mu(t)-r(t)) \Lambda(t)y - \frac{\Lambda(t)^2}{4 \Theta(t)}\pi^2 \sigma(t)^2\bigg] \\
	\label{mmvPC.hjbi.equation.compact.guess.2} &  + \sup_{u \geq 0} \bigg[\mu_0 \kappa_r u \Lambda(t)y -  \frac{\Lambda(t)^2}{4 \Theta(t)}u^2 \sigma_0^2\bigg].
	\end{align}
	
	\par Secondly, we try to solve
	\begin{equation*}
	\begin{cases}
	\sup_{\pi \in \mathbb{R}} \bigg[\pi (\mu(t)-r(t)) \Lambda(t)y - \frac{\Lambda(t)^2}{4 \Theta(t)}\pi^2 \sigma(t)^2\bigg], \\
	\sup_{u \geq 0} \bigg[\mu_0 \kappa_r u \Lambda(t)y -  \frac{\Lambda(t)^2}{4 \Theta(t)}u^2 \sigma_0^2\bigg].
	\end{cases}
	\end{equation*}
	
	\par If $\frac{\Lambda(t)^2}{\Theta(t)} > 0$, then the maximum is attained at
	\begin{equation}
	\label{mmvPC.u.maximizer}
	\begin{cases}
	\hat{\pi}(t,x,y) = \frac{2y  \Theta(t)}{\Lambda(t)} \frac{\mu(t)-r(t)}{\sigma(t)^2} , \\
	\hat{u}(t,x,y) = \frac{2y \Theta(t)}{\Lambda(t)} \frac{\mu_0 \kappa_r}{\sigma_0^2}.
	\end{cases}
	\end{equation}
	Substituting \eqref{mmvPC.u.maximizer} into \eqref{mmvPC.hjbi.equation.compact.guess.2}, we get the following three ordinary differential equations
	\begin{align*}
	\begin{cases}
	\Lambda'(t) + r(t)\Lambda(t) = 0, \\
	\Theta'(t) + \rho(t) \Theta(t) = 0, \\
	\Psi'(t) + \mu_0(\kappa-\kappa_r)\Lambda(t) = 0, \\
	\end{cases}
	\end{align*}
	with terminal conditions $\Lambda(T)=1$, $\Theta(T) = \frac{1}{2\theta}$ and $\Psi(T) = 0$. It is easy to find that the solutions of the above equations are as follows
	\begin{align}
	\label{mmvPC.guess.coefficient}
	\begin{cases}
	\Lambda(t) = \exp \left(\int_t^T r(s) ds \right), \\
	\Theta(t) = \frac{1}{2 \theta} \exp \left(\int_t^T \rho(s) ds \right), \\
	\Psi(t) = \mu_0(\kappa - \kappa_r) \int_t^T  \exp \left(\int_s^T r(v) dv \right)ds.\\
	\end{cases}
	\end{align}
	
	\par  Since $\Theta(t)$ and $\Lambda(t)$ are both positive functions, the above requirements are all satisfied. By plugging \eqref{mmvPC.guess.coefficient} into \eqref{mmvPC.u.maximizer} and \eqref{mmvPC.q.minimizer}, the saddle point for the non-linear PDE \eqref{mmvPC.hjbi.equation.compact} is proved to be \eqref{mmvPC.saddle}. Moreover, it is easy to verify that \eqref{mmvPC.saddle} is indeed the saddle point of the HJBI equation \eqref{mmvPC.hjbi.equation} since it satisfies
	\begin{align*}
	& \mathcal{L}_t^{\hat{a}(t,x,y),\hat{b}(t,\cdot,x,y)}\phi(t,x,y) = 0, \\
	& \mathcal{L}_t^{\hat{a}(t,x,y),b}\phi(t,x,y) \geq 0, \qquad \forall b \in \mathbb{R}^2, \\
	& \mathcal{L}_t^{a,\hat{b}(t,\cdot,x,y)}\phi(t,x,y) =  0, \qquad \forall a \in \mathbb{R} \times \mathbb{R}_+.
	\end{align*}
	In the end, by inserting \eqref{mmvPC.guess.coefficient} into \eqref{mmvPC.guess}, \eqref{mmvPC.value.function} is proved to be a smooth solution of \eqref{mmvPC.hjbi.equation}.
\end{proof}

\par Define
\begin{equation}
	\label{mmvPC.nash.equilibrium}
	\begin{cases}
		&a^*(t;s,x,y) = \hat{a}(t,X^{\hat{a}}(t;s,x,y),Y^{\hat{b}}(t;s,x,y)), \\
		&b^*(t,z;s,x,y) = \hat{b}(t,z,X^{\hat{a}}(t;s,x,y),Y^{\hat{b}}(t;s,x,y)), \\
	\end{cases}
\end{equation}
where $(s,x,y)$ represents that $X(s) =x$, $Y(s) = y$.

\begin{theorem} \label{mmvPC.theorem.verification}
\par {\bf (verification theorem)} $\phi(t,x,y)$ in \eqref{mmvPC.value.function} is the value function of Problem {\bf (P$_{sxy}$)}. $(a^*, b^*)$ is a Nash equilibrium of Problem {\bf (P$_{sxy}$)}.
\end{theorem}

\par We first give a lemma verifying the admissibility of $a^*(t;s,x,y)$ and $b^*(t,z;s,x,y)$ before proving this theorem.
\begin{lemma} \label{mmvPC.lemma.admissible.check}
	The strategies $a^*(t;s,x,y)$ and $b^*(t,z;s,x,y)$ defined by \eqref{mmvPC.nash.equilibrium} belong to $\mathcal{A}[s,T]$ and $\mathcal{B}[s,T]$, respectively.
\end{lemma}

\begin{proof}
	\par By Theorem 13 (2) and Lemma 7 (2) of \cite{kabanov1979absolute}, SDE \eqref{mmvPC.dynamic.y} has a unique solution in the class of nonnegative local martingales which we also denote as $Y$. Since $b^*(t,z;s,x,y)$ is deterministic, bounded and satisfies
	\begin{align}
	q^*(t,z;s,x,y) \geq -1,
	\end{align}
	$Y$ is a square-integrable martingale. Consequently, $b^*(t,z;s,x,y)\in \mathcal{B}[s,T]$.
	\par On the other hand, since
	\begin{equation*}
	\sup_{s \leq t \leq T} \mathbb{E}^PY(t)^2 < +\infty,
	\end{equation*}
	we can deduce that
	\begin{equation*}
	\mathbb{E}^P\bigg[ \int_s^T \pi^*(t;s,x,y)^2 + u^*(t;s,x,y)^2 dt\bigg] < +\infty,
	\end{equation*}
	which proves that $a^*(t;s,x,y)\in \mathcal{A}[s,T]$.
\end{proof}

\par Proof of Theorem \ref{mmvPC.theorem.verification} is as follows.
\begin{proof}
	\par By applying It{\^o}'s Lemma to function $\phi(t,x,y)$ in \eqref{mmvPC.value.function}, for any admissible strategies, we have
	\begin{align*}
	&\phi(T,X(T),Y(T)) = \phi(s,x,y) + \int_s^T \mathcal{L}_t^{a,b} \phi(t,X(t),Y(t)) dt \\
	&+ \int_s^T \phi_x \pi(t) \sigma(t) dW(t) + \int_s^T \phi_y Y(t) p(t) dW(t) \\
	&+ \int_s^T \int_{K} \phi(t,X(t-)-u(t)z,Y(t-)(1+q(t,z)))-\phi(t,X(t-),Y(t-)) \widetilde{N}(dt,dz).
	\end{align*}
	\par Let $\tau_B = \inf\{t > 0: |\phi_x| \geq B \} \wedge \inf\{t > 0: |\phi_y Y(t)| \geq B \} \wedge \inf\{t > 0: \int_s^t \int_{K} |\phi(t,X(r-)-u(r)z,Y(r-)(1+q(r,z)))| v(dz)dr \geq B \}$. Recall that \eqref{mmvPC.condition.integrable.property} holds, therefore
	\begin{equation*}
	\mathbb{E}^P \bigg[ \phi(T\vee \tau_B,X(T\vee \tau_B),Y(T\vee \tau_B)) \bigg] = \phi(s,x,y) + \mathbb{E}^P \int_s^{T \vee \tau_B} \mathcal{L}_t^{a,b} \phi(t,X(t),Y(t)) dt.
	\end{equation*}
	\par By Lemma \ref{mmvPC.lemma.admissible.check}, the optimal pair $(a^*(t;s,x,y), b^*(t,z;s,x,y))$ given by \eqref{mmvPC.nash.equilibrium} belongs to $\mathcal{A}[s,T] \times \mathcal{B}[s,T]$, thus is admissible. Since $\phi$ is the solution of the HJBI equation \eqref{mmvPC.hjbi.equation}, substituting $(a^*(t;s,x,y), b^*(t,z;s,x,y))$ into it gives
	\begin{align*}
	\begin{cases}
	&\mathbb{E}^P \bigg[ \phi(T\vee \tau_B,X^{a^*,b^*}(T\vee \tau_B),Y^{a^*,b^*}(T\vee \tau_B)) \bigg] = \phi(s,x,y), \\
	&\mathbb{E}^P \bigg[ \phi(T\vee \tau_B,X^{a^*,b}(T\vee \tau_B),Y^{a^*,b}(T\vee \tau_B)) \bigg] \geq \phi(s,x,y), \\
	&\mathbb{E}^P \bigg[ \phi(T\vee \tau_B,X^{a,b^*}(T\vee \tau_B),Y^{a,b^*}(T\vee \tau_B)) \bigg] \leq \phi(s,x,y).
	\end{cases}
	\end{align*}
	\par Letting $B \to \infty$ we obtain
	\begin{equation*}
	J^{a,b^*}(s,x,y) \leq J^{a^*,b^*}(s,x,y) = \phi(s,x,y) \leq J^{a^*,b}(s,x,y).
	\end{equation*}
	\par Consequently, by Lemma \ref{mmvPC.lemma.1},
	\begin{equation*}
	\phi(s,x,y) = \sup_{a  \in \mathcal{A}[s,T]} \left ( \inf_{b \in \mathcal{B}[s,T]}J^{a,b}(s,x,y) \right ) = \inf_{b \in \mathcal{B}[s,T]} \left ( \sup_{a  \in \mathcal{A}[s,T]} J^{a,b}(s,x,y) \right ),
	\end{equation*}
	and $(a^*(t;s,x,y), b^*(t,z;s,x,y))$ is a Nash equilibrium of Problem {\bf (P$_{sxy}$)}.
\end{proof}

\begin{lemma} \label{mmvPC.lemma.x.y}
\par Let $X^{a^*}(t)$ and $Y^{b^*}(t)$ be the state processes under the Nash equilibrium $(a^*, b^*)$, then
\begin{align}
\label{mmvPC.equation.wealth.P}
\frac{1}{\theta}   Y^{b^*}(t) e^{  \int_t^T \rho(u) - r(u)du  } = x e^{ \int_s^t r(s) ds } + \int_s^t \mu_0(\kappa -\kappa_r) e^{ \int_v^t r(u) du } dv - X^{a^*}(t) + \frac{1}{\theta}  y e^{  \int_s^T \rho(u) du } e^{ -\int_t^T r(s) ds }.
\end{align}
\end{lemma}
\begin{proof}
\par Since $X^a$ is an Ornstein-Uhlenbeck process, it has an explicit solution as follows
\begin{align*}
X^{a^*}(t) =& x e^{ \int_s^t r(s) ds } + \int_s^t \mu_0(\kappa -\kappa_r) e^{ \int_v^t r(u) du } dv \\
&+ \frac{1}{\theta} e^{ -\int_t^T r(u) du } \bigg\{\int_s^t (\frac{\mu(v)-r(v)}{\sigma(v)})^2 e^{  \int_v^T \rho(u)du  } Y^{b^*}(v) dv + \int_s^t(\frac{\mu_0 \kappa_r }{\sigma_0} )^2 e^{  \int_v^T \rho(u)du  } Y^{b^*}(v) dv \\
&+ \int_s^t \frac{\mu(v)-r(v)}{\sigma(v)}e^{  \int_v^T \rho(u)du  } Y^{b^*}(v) dW(v) - \int_s^t \int_{K} \frac{\mu_0 \kappa_r z}{ \sigma_0^2} e^{  \int_v^T \rho(u)du  }  Y^{b^*}(v) \widetilde{N}(dv,dz) \bigg\} \\
=& x e^{ \int_s^t r(s) ds } + \int_s^t \mu_0(\kappa -\kappa_r) e^{ \int_v^t r(u) du } dv \\
&+ \frac{1}{\theta} e^{ -\int_t^T r(u) du } \bigg\{\int_s^t \rho(v) e^{  \int_v^T \rho(u)du  } Y^{b^*}(v) dv - \int_s^t e^{  \int_v^T \rho(u)du  } dY^{b^*}(v) \bigg\} \\
=& x e^{ \int_s^t r(s) ds } + \int_s^t \mu_0(\kappa -\kappa_r) e^{ \int_v^t r(u) du } dv \\
&- \frac{1}{\theta} e^{ -\int_t^T r(u) du } \bigg\{\int_s^t Y^{b^*}(v) de^{  \int_v^T \rho(u)du  } + \int_s^t e^{  \int_v^T \rho(u)du  } dY^{b^*}(v) \bigg\} \\ 
=& x e^{ \int_s^t r(s) ds } + \int_s^t \mu_0(\kappa -\kappa_r) e^{ \int_v^t r(u) du } dv - \frac{1}{\theta} e^{ -\int_t^T r(u) du } \bigg\{ Y^{b^*}(t) e^{  \int_t^T \rho(u)du  } - y e^{  \int_s^T \rho(u)du  } \bigg\}.
\end{align*}	
\end{proof}

\begin{theorem} \label{mmvPC.strategies.optimal}
\par If the initial values are $X^{a^*}(s) = x$ and $Y^{b^*}(s)=y$, the optimal strategies of Problem {\bf (MMV$_{\theta}$)} are given by
\begin{align}
\label{mmvPC.optimal_strategy}
\begin{cases}
\pi^*(t;s,x,y) =  \frac{\mu(t)-r(t)}{\sigma(t)^2} \biggl\{x e^{ \int_s^t r(s) ds } + \int_s^t \mu_0(\kappa -\kappa_r) e^{ \int_v^t r(u) du } dv - X^{a^*}(t) + \frac{1}{\theta}  y e^{  \int_s^T \rho(v)dv } e^{ -\int_t^T r(s) ds }\biggr\}, \\
u^*(t;s,x,y) = \frac{\mu_0 \kappa_r }{ \sigma_0^2} \biggl\{x e^{ \int_s^t r(s) ds } + \int_s^t \mu_0(\kappa -\kappa_r) e^{ \int_v^t r(u) du } dv - X^{a^*}(t) + \frac{1}{\theta}  y e^{  \int_s^T \rho(v)dv } e^{ -\int_t^T r(s) ds }\biggr\},
\end{cases}
\end{align}
and the value function of Problem {\bf (MMV$_{\theta}$)} is given by
\begin{align}
\nonumber \Phi_{\theta} =& \exp \left(\int_s^T r(v) dv \right) xy + \frac{1}{2 \theta} \exp \label{mmvPC.value_function_final} \left(\int_s^T \rho(v) dv \right) y^2\\ &+  \mu_0(\kappa - \kappa_r) y \int_s^T  \exp \left(\int_v^T r(u) du \right)dv - \frac{1}{2\theta}.
\end{align}
\end{theorem}
{\color{blue}
	\begin{remark}
		By taking the conditional expectation $ \mathbb{E}_{s,x,y}^P := \mathbb{E}^P[\cdot|X(s)=x,Y(s)=y]$ on \eqref{mmvPC.equation.wealth.P}, we have
		\begin{align*}
		\mathbb{E}_{s,x,y}^PX^{a^*}(t) + \frac{1}{\theta}   y e^{  \int_t^T \rho(u) - r(u)du  } = x e^{ \int_s^t r(s) ds } + \int_s^t \mu_0(\kappa -\kappa_r) e^{ \int_v^t r(u) du } dv + \frac{1}{\theta}  y e^{  \int_s^T \rho(u) du } e^{ -\int_t^T r(s) ds }.
		\end{align*}
		Then the optimal strategies in \eqref{mmvPC.optimal_strategy} can be written as
		\begin{align*}
		\begin{cases}
		\pi^*(t;s,x,y) =  \frac{\mu(t)-r(t)}{\sigma(t)^2} \biggl\{\mathbb{E}_{s,x,y}^PX^{a^*}(t) + \frac{1}{\theta}   y e^{  \int_t^T \rho(u) - r(u)du  } - X^{a^*}(t) \biggr\}, \\
		u^*(t;s,x,y) = \frac{\mu_0 \kappa_r }{ \sigma_0^2} \biggl\{\mathbb{E}_{s,x,y}^PX^{a^*}(t) + \frac{1}{\theta}   y e^{  \int_t^T \rho(u) - r(u)du  } - X^{a^*}(t) \biggr\}.
		\end{cases}
		\end{align*}
		Here $\frac{\mu(t)-r(t)}{\sigma(t)^2}$ and $\frac{\mu_0 \kappa_r }{ \sigma_0^2}$ are scale factors called the market price of risk. The optimal strategies have the following economic explanation: The insurer has a benchmark $\mathbb{E}_{s,x,y}^PX^{a^*}(t) + \frac{1}{\theta}   y e^{  \int_t^T \rho(u) - r(u)du  }$ which is the expected wealth $\mathbb{E}_{s,x,y}^PX^{a^*}(t)$ plus an additional anticipated return $\frac{1}{\theta}   y e^{  \int_t^T \rho(u) - r(u)du  }$. The larger the parameter of risk aversion $\theta$, the smaller the additional anticipated return. The insurer would compare the current wealth $X^{a^*}(t)$ to the benchmark, and choose its strategies according to the difference modified by the scale factors. If the current wealth is smaller than the benchmark, the insurer would take a more risky strategy and if the current wealth is large, the insurer would be more conservative.
	\end{remark}
}

{\color{red}
\begin{remark}
	For the no-jump case, where the claim process is no longer a compound Poisson process but are assumed to satisfies the dynamics
	\begin{equation}
	\label{mmvPC.diffusion_approximation}
	d\widetilde{C}(t) := \int_{K} z v(dz) dt - \int_{K} z^2 v(dz) dW_0(t),
	\end{equation}
	where $W_0(t)$ is a standard Brownian motion independent of $W(t)$, and $v(\cdot)$ is the L\'{e}vy measure defined in Subsection \ref{mmvPC.section.market}, it can be calculated that for the new problem the value functions is also \eqref{mmvPC.value_function_final} (see \cite{li2021optimal})). In this case, \eqref{mmvPC.diffusion_approximation} is called the diffusion approximation of the compound Poisson process $C(t) = \int_0^t \int_{K} z N(ds,dz)$.
\end{remark}
}

\section{Efficient frontier}



\par We now discuss the relation between the variance and the mean of the terminal wealth process for varying $\theta \in [0,+\infty)$. Denote $a^{\theta}$ as the optimal strategy for Problem {\bf (MMV$_{\theta}$)}. If we choose a strategy $a(t)  \equiv (0,0)$, then by substituting it into SDE \eqref{mmvPC.wealth}, we can get the riskless wealth of the insurance company:
\begin{equation}
	\label{mmvPC.x0}
	X^{(0,0)}(t) =  x_0 e^{ \int_0^t r(s) ds } + \int_0^t \mu_0(\kappa -\kappa_r) e^{ \int_s^t r(u) du } ds.
\end{equation}
Here, we denote $X^{(0,0)} = X^{(0,0)}(T)$.

\begin{theorem} \label{mmvPC.theorem.efficientfrontier}
\par If we let the parameter $\theta \in [0,+\infty)$ be arbitrary, then the variance of the terminal wealth process can be written as a functional of the expectation as follows
\begin{equation*}
	\mathbb{V}ar^PX(T) =
	\begin{cases}
		 \frac{ \exp(- \int_0^T \rho(s) ds) }{1 -\exp(- \int_0^T \rho(s) ds)}\left(\mathbb{E}^PX(T)- X^{(0,0)}  \right)^2, &\text{ if  } \quad \mathbb{E}^PX(T) > X^{(0,0)}, \\
		 0, &\text{ if  } \quad	\mathbb{E}^PX(T) \leq X^{(0,0)},
	\end{cases}
\end{equation*}
where $\rho(t) = \frac{\mu_0^2 \kappa_r^2}{\sigma_0^2} + b(t)^T \Sigma(t)^{-1} b(t)$. The corresponding strategy is given by
\begin{equation*}
a^{\theta}=
\begin{cases}
(\pi^*,u^*) \text{ in Corollary \ref{mmvPC.strategies.optimal}}, &\text{ if  } \quad \mathbb{E}^PX(T) > X^{(0,0)}, \\
(0,0), &\text{ if  } \quad \mathbb{E}^PX(T) \leq X^{(0,0)}.
\end{cases}
\end{equation*}
\end{theorem}

{\color{magenta}
\begin{remark}
\par The expression in Theorem \ref{mmvPC.theorem.efficientfrontier} coincides with the efficient frontier of the classical MV model (see Theorem 8 of \cite{bi2014dynamic}), which means the optimal strategies with different $\theta$ in this paper also possess the property of the efficient strategies for the classical MV model. Specifically, it is well known that the efficient strategy of the classical MV model is an ``optimal" strategy, which minimizes the variance while ensuring that the mean is fixed, and for our MMV problem, if we fix $\mathbb{E}^PX(T) := \xi \geq X^{(0,0)}$, its optimal strategy is \eqref{mmvPC.optimal_strategy} with
\begin{equation*}
\theta = \frac{\exp(\int_0^T \rho(s)ds) - 1}{\xi-X^{(0,0)}},
\end{equation*}
and, under this strategy, the variance $\mathbb{V}ar^PX(T)$ indeed attains its minimum $\frac{ \exp(- \int_0^T \rho(s) ds) }{1 -\exp(- \int_0^T \rho(s) ds)}\left(\xi- X^{(0,0)}  \right)^2$. In view of this, we call
\begin{equation*}
\label{mmvPC.ef.1}
\begin{cases}
\left( \frac{ \exp(- \int_0^T \rho(s) ds) }{1 -\exp(- \int_0^T \rho(s) ds)}\left(\mathbb{E}^PX(T)- X^{(0,0)}  \right)^2, \mathbb{E}^PX(T) \right), &\text{ if  } \quad \mathbb{E}^PX(T) > X^{(0,0)}, \\
(0, X^{(0,0)}), &\text{ if  } \quad	\mathbb{E}^PX(T) \leq X^{(0,0)},
\end{cases}
\end{equation*}
the efficient frontier and \eqref{mmvPC.optimal_strategy} with $\theta := \frac{\exp(\int_0^T \rho(s)ds) - 1}{\xi-X^{(0,0)}}$ the efficient strategies for the MMV problem.
\end{remark}
}

\begin{remark}
	\par Let $\mathbb{E}^PX^{a^{\theta}}(T) = \xi$, then by equality \eqref{mmvPC.EPx}, we have
	\begin{equation}
	\label{mmvPC.theta}
	-\frac{1}{\theta} = \frac{x_0 \exp(\int_0^T r(s) ds) + \mu_0 (\kappa - \kappa_r) \int_0^T \exp(\int_s^Tr(v)dv) ds-\xi}{\exp(\int_0^T \rho(s)ds) - 1},
	\end{equation}
	which agrees with $\beta^*$ in Theorem 8 of \cite{bi2014dynamic}. Since, in \cite{bi2014dynamic}, $\beta^*$ is the Lagrange multiplier in the variance-minimization problem, $-\frac{1}{\beta^*}$ plays the role of risk aversion coefficient.
\end{remark}

\par We first give some lemmas before proving Theorem \ref{mmvPC.theorem.efficientfrontier}.

\begin{lemma} \label{mmvPC.lemma.d1}
	\par If $X(0) = x_0$, $Y(0) = 1$, we have, for the optimal strategies $a^*$, $b^*$ and $\frac{dQ^*}{dP} \big|_{\mathcal{F}_t} = Y^{a^*}(t)$, 
	\begin{align}
	\label{mmvPC.EPx} &\mathbb{E}^PX^{a^*}(t) =   x_0 e^{ \int_0^t r(u) du } + \int_0^t \mu_0(\kappa -\kappa_r) e^{ \int_s^t r(u) du } ds + \frac{1}{\theta} e^{-\int_t^T r(s) ds } \biggl\{  e^{ \int_0^T \rho(s) ds } - e^{ \int_t^T \rho(s) ds }\biggr\}, \\
	\label{mmvPC.EQy} &\mathbb{E}^{Q^*} \bigg[ \frac{dQ^*}{dP}\bigg|_{\mathcal{F}_t} \bigg] \equiv \mathbb{E}^P \bigg[ \frac{dQ^*}{dP}\bigg|_{\mathcal{F}_t} \bigg]^2 \equiv \mathbb{E}^P (Y^{b^*}(t))^2 = \exp \left( \int_0^t \rho(s) ds \right).
	\end{align}
\end{lemma}
\begin{proof}
	\par \eqref{mmvPC.EPx} is obvious by taking expectation on both sides of \eqref{mmvPC.equation.wealth.P} under $P$. \eqref{mmvPC.EQy} is obtained by applying It{\^o}'s Lemma to $Y(t)^2$.
\end{proof}

\begin{lemma} \label{mmvPC.lemma.Var.E.relationship}
	\par For any fixed $\theta \in [0,+\infty)$, the variance of the wealth process for the optimal strategy $u^* $ is as follows
	\begin{equation}
	\label{mmvPC.var.1} \mathbb{V}ar^PX^{a^*}(t) = \frac{ \exp(- \int_0^t \rho(s) ds) }{1 -\exp(- \int_0^t \rho(s) ds)}\bigg(\mathbb{E}^PX^{a^*}(t)- \mathbb{E}^{Q^*}X(t)  \bigg)^2,
	\end{equation}
	with $Q^*$ defined by $\frac{dQ^*}{dP} \big|_{\mathcal{F}_t} = Y^{a^*}(t)$ and
	\begin{equation}
	\label{mmvPC.EQx} \mathbb{E}^{Q^*}X(t) = x_0 e^{ \int_0^t r(s) ds } + \int_0^t \mu_0(\kappa -\kappa_r) e^{ \int_s^t r(u) du } ds \equiv X^{(0,0)}(t).
	\end{equation}
\end{lemma}
\begin{proof}
	\par By Theorem \ref{mmvPC.lemma.x.y} and \eqref{mmvPC.EQy}, we have
	\begin{align*}
	\mathbb{V}ar^PX^{a^*}(t) =& \frac{1}{\theta^2}  e^{  2 \int_t^T( \rho(s) - r(s))ds  } \mathbb{V}ar^P  Y^{b^*}(t) \\
	=& \frac{1}{\theta^2}  e^{  2 \int_t^T( \rho(s) - r(s))ds  } \bigg(\mathbb{E}^P  Y^{b^*}(t)^2 - (\mathbb{E}^P Y^{b^*}(t))^2\bigg) \\
	=& \frac{1}{\theta^2}  e^{  2 \int_t^T( \rho(s) - r(s))ds  } \bigg(\mathbb{E}^{Q^*}  Y^{b^*}(t) - \mathbb{E}^P Y^{b^*}(t)\bigg) \\
	=& \frac{1}{\theta^2}  e^{  2 \int_t^T( \rho(s) - r(s))ds  } (e^{ \int_0^t \rho(s) ds } - 1 ),
	\end{align*}
	and
	\begin{align*}
	\mathbb{E}^PX^{a^*}(t) - \mathbb{E}^{Q^*}X(t) = \frac{1}{\theta} (e^{ \int_0^t \rho(s) ds } - 1 ) e^{ \int_t^T (\rho(s) -r(s)) ds }.
	\end{align*}
	
	\par Hence
	\begin{align*}
	\mathbb{V}ar^PX^{a^*}(t) = \frac{1}{e^{ \int_0^t \rho(s) ds } - 1 } \bigg(\mathbb{E}^PX^{a^*}(t)- \mathbb{E}^{Q^*}X(t)  \bigg)^2,
	\end{align*}
	which proves \eqref{mmvPC.var.1}.
\end{proof}

\par Proof of Theorem \ref{mmvPC.theorem.efficientfrontier} is as follows.
\begin{proof}
	\par By taking expectation of \eqref{mmvPC.dynamic.x} and comparing it with \eqref{mmvPC.x0}, we have $\mathbb{E}^PX^a(T) \leq X^{(0,0)}$ if and only if
	\begin{equation}
	\label{mmvPC.constraint}
	\mathbb{E}^P \bigg[ \int_0^T e^{ \int_s^T r(u) du } \pi(s)  (\mu(s)-r(s)) ds + \int_0^T \mu_0 \kappa_r e^{ \int_s^T r(u) du } u(s) ds \bigg] \leq 0.
	\end{equation}
	For the case $\mathbb{E}^PX^a(T) \leq X^{(0,0)}$, we claim that $(a^{\theta} , b^{\theta} ) = (\vec{0}, \vec{0})$ is the Nash equilibrium of Problem {\bf (P$_{0x_01}$)} under constraint \eqref{mmvPC.constraint}. We obtain
	\begin{equation*}
	Y^0(t) \equiv 1, \quad \forall t \in[0,T],
	\end{equation*}
	which gives
	\begin{equation*}
	J^{0,0}(0,x_0,1) = \mathbb{E}^P\bigg[X^{(0,0)}(T)Y^0(T)+\frac{1}{2\theta}(Y^0(T))^2\bigg] = X^{(0,0)}(T) + \frac{1}{2\theta},
	\end{equation*}
	where $X^{(0,0)}(T)$ is deterministic that is given by \eqref{mmvPC.x0}. For any $b  \in \mathcal{B}[0,T]$,
	\begin{align*}
	J^{0,b}(0,x_0,1) =& \mathbb{E}^P\bigg[X^{(0,0)}(T)Y^b(T)+\frac{1}{2\theta}(Y^b(T))^2\bigg] \\
	=& X^{(0,0)}(T) + \frac{1}{2\theta} \exp \left( \int_0^t \bigg[p(s)^2 + \int_{K}  q(s,z)^2 v(dz) \bigg]  ds \right) \\
	\geq& X^{(0,0)}(T) + \frac{1}{2\theta} = J^{0,0}(0,x_0,1).
	\end{align*}
	For any $a  \in \mathcal{A}[0,T]$ satisfying the constraint \eqref{mmvPC.constraint},
	\begin{align*}
	J^{a,0}(0,x_0,1) =& \mathbb{E}^P[X^a(T)Y^0(T)+\frac{1}{2\theta}(Y^0(T))^2] \\
	=& \mathbb{E}^P\bigg[X^a(T)\bigg]+\frac{1}{2\theta} \\
	=& X^{(0,0)}(T) + \frac{1}{2\theta} \\
	&+ \mathbb{E}^P \bigg[ \int_0^T e^{ \int_s^T r(u) du } \pi(s)  (\mu(s)-r(s)) ds + \int_0^T \mu_0 \kappa_r e^{ \int_s^T r(u) du } u(s) ds \bigg] \\
	\leq& X^{(0,0)}(T) + \frac{1}{2\theta} = J^{0,0}(0,x_0,1).
	\end{align*}
	By Lemma \ref{mmvPC.lemma.1}, for the case $\mathbb{E}^PX(T) \leq X^{(0,0)}$, the efficient strategy is $a^{\theta} = (0,0)$. Noting that $\mathbb{V}ar^PX^{(0,0)}(T) = 0$ and $\mathbb{E}^PX^{(0,0)}(T) = X^{(0,0)}(T) \equiv X^{(0,0)}$, the efficient frontier is the point $(0, X^{(0,0)})$. On the other hand,  since the optimal strategy $a^*(t) > 0$, for $\forall t \in [0,T]$ and
	\begin{equation*}
	\mathbb{E}^P \bigg[ \int_0^T e^{ \int_s^T r(u) du } \pi^*(s)  (\mu(s)-r(s)) ds + \int_0^T \mu_0 \kappa_r e^{ \int_s^T r(u) du } u^*(s) ds \bigg] > 0,
	\end{equation*}
	we have $\mathbb{E}^PX^{a^*}(T) > X^{(0,0)}$. For this case, the efficient strategy is given by $a^{\theta}=(\pi^*,u^*)$ in Theorem \ref{mmvPC.strategies.optimal}. The efficient frontier is given by \eqref{mmvPC.var.1}.
\end{proof}

\begin{corollary} \label{mmvPC.corollary.3}
	\par If there is no investment in the model, i.e., $\mu = \sigma = r = 0$, the efficient frontier is
	\begin{equation*}
		\label{mmvPC.ef.2}
		\begin{cases}
			\left( \frac{ \exp(- \int_0^T \rho(s) ds) }{1 -\exp(- \int_0^T \rho(s) ds)}\left(\mathbb{E}^PX(T)- X^{(0,0)}_{niv}  \right)^2, \mathbb{E}^PX(T) \right), &\text{ if }  \mathbb{E}^PX(T) > X^{(0,0)}_{niv}, \\
			(0, X^{(0,0)}_{niv}), &\text{ if } 	\mathbb{E}^PX(T) \leq X^{(0,0)}_{niv},
		\end{cases},
	\end{equation*}
	in which
	\begin{equation*}
		X^{(0,0)}_{niv} = x_0 + \mu_0(\kappa -\kappa_r)T.
	\end{equation*}
\end{corollary}

\begin{remark}
	\par The efficient frontier in Corollary \ref{mmvPC.corollary.3} coincides with that for the classical MV model (see \cite{bauerle2005benchmark}).
\end{remark}

\begin{corollary} \label{mmvPC.corollary.4}
\par If there is no insurance in the model, in other words, $\lambda = v(dx) = 0$, the efficient frontier is
\begin{equation*}
	\label{mmvPC.ef.3}
	\begin{cases}
		 \left( \frac{ \exp(- \int_0^T \rho(s) ds) }{1 -\exp(- \int_0^T \rho(s) ds)}\left(\mathbb{E}^PX(T)- X^{(0,0)}_{nis}  \right)^2, \mathbb{E}^PX(T) \right), &\text{ if }  \mathbb{E}^PX(T) > X^{(0,0)}_{nis}, \\
		 (0, X^{(0,0)}_{nis}), &\text{ if } 	\mathbb{E}^PX(T) \leq X^{(0,0)}_{nis},
	\end{cases},
\end{equation*}
in which
\begin{equation*}
	X^{(0,0)}_{nis} = x_0 \exp \left( \int_0^T r(s) ds \right).
\end{equation*}
\end{corollary}



\section{Numerical example}

\begin{example}
	\label{example.nonmonotone}
	\par Consider a company whose risk-aversion is $\theta = 2$. The company's deterministic profit from the operation is $\eta \equiv 10$. Assume there exists an opportunity of arbitrage in the financial market and the return is $\varepsilon \sim U(0,12)$, an uniformly distributed random variable. Then, the company's total wealth depends on whether it exploits the opportunity of arbitrage, namely, the total wealth is $X^u = \eta + u \varepsilon$ where $u = 0$ or $1$. It is obvious that $X^1 > X^0$, but
	\begin{align*}
	U_{2}(X^0) = 10,
	\end{align*}
	which is greater than
	\begin{align*}
	U_{2}(X^1) = 10 + \mathbb{E}^P \varepsilon - \frac{\theta}{2} \mathbb{V}ar^P \varepsilon = 4.
	\end{align*}
	It shows that the classical MV criterion may mislead the investors in certain situation. The MMV criterion remedies this problem and we have 
	\begin{align*}
	V_{2}(X^0) = 10,
	\end{align*}
	which is less than
	\begin{align*}
	V_{2}(X^1) = \frac{39}{4}+\frac{4\sqrt{3}}{3} \approx 12.0594.
	\end{align*}
\end{example}

\begin{figure}[htb]
	\centering
	\includegraphics[width=5in]{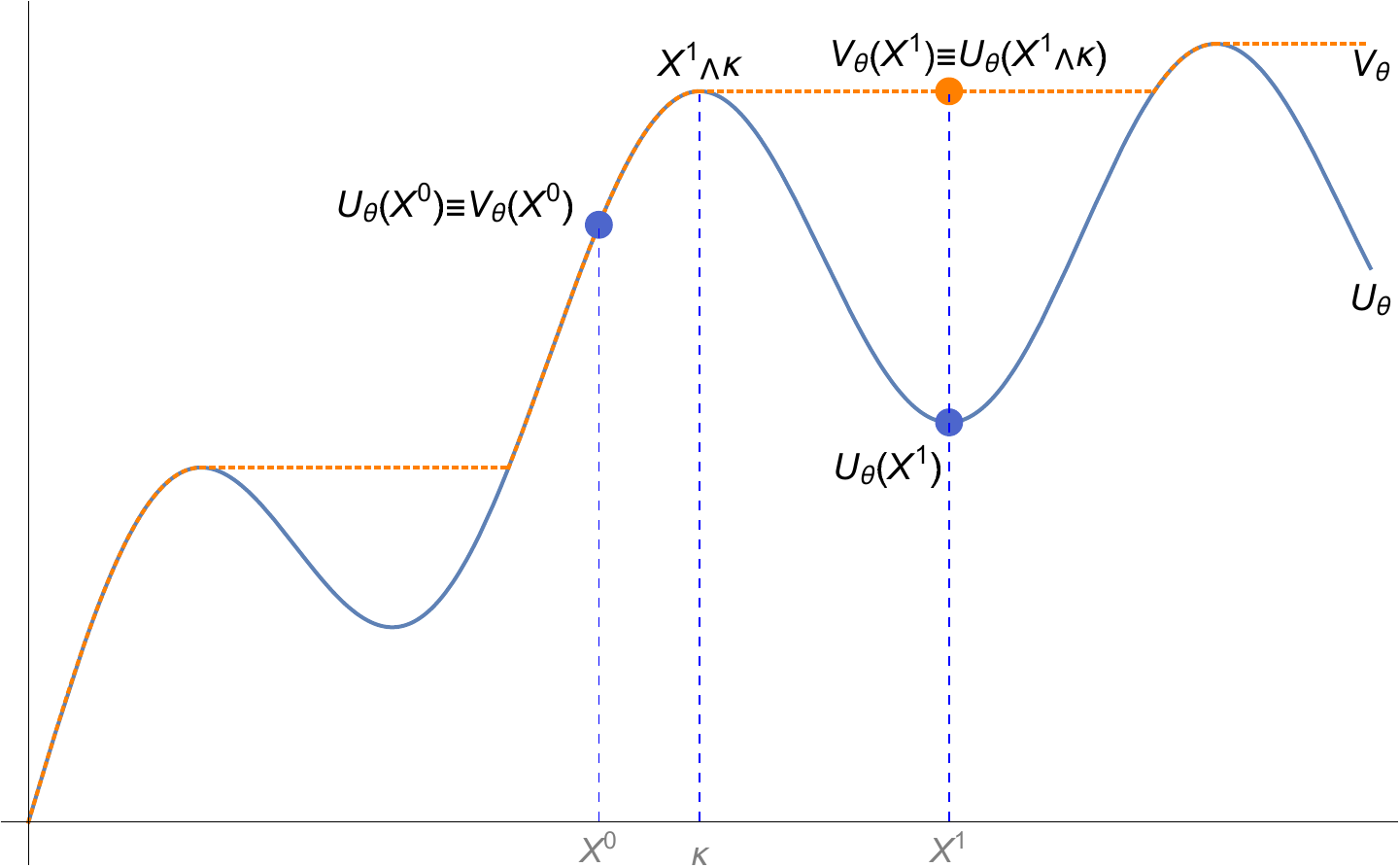}
	\caption{How does the MMV remedy the defect of non-monotonicity of classical MV}
	\label{mmvPC.fig.monotonicity}
\end{figure}

\par The value of $V_{2}(X^1)$ is calculated with the aid of Theorem 2.2 in \cite{maccheroni2009portfolio}. Let $\mathcal{G}_{2}$ be the domain of monotonicity of the classical MV criterion $U_2$. Since $X^1 \notin \mathcal{G}_{2}$, we have
\begin{align*}
V_2(X^1) = U_2(X^1 \wedge \kappa),
\end{align*}
where
\begin{align*}
\kappa = \max\{t \in  \mathbb{R},X^1 \wedge t \in \mathcal{G}_{2} \}.
\end{align*}
It is easy to obtain that $\kappa = 10 + 2 \sqrt{3}$. Then
\begin{align*}
V_{2}(X^1) = \mathbb{E}[X^1 \wedge \kappa] - \frac{2}{2} \mathbb{V}ar[X^1 \wedge \kappa] = \frac{19}{2} + 2\sqrt{3} - (-\frac{1}{4} + \frac{2\sqrt{3}}{3}) = \frac{39}{4}+\frac{4\sqrt{3}}{3}.
\end{align*}
\par The monotonicity of the MMV criterion can be illustrated by Figure \ref{mmvPC.fig.monotonicity}. The horizontal axis represents an ordering and sequencing set of random variables where the order $\preccurlyeq$ is defined by $\mathop{\leq}\limits^{\scriptscriptstyle P-a.s.}$. The vertical axis is the value of the criterion. The blue line illustrates the value of the classical MV criterion for different random variables. The orange dashed line illustrates the value of the MMV criterion. Obviously, the blue line is not always increasing and, in the domain of non-monotonicity, it is remedied by the orange dashed line. 

\begin{experiment}
\par Consider an insurance company who wants to optimize its wealth in three years under the MMV criterion. For simplicity, we assume the amount of claim is exponentially distributed with intensity $\delta$. The other coefficients in the model are assumed to be
\begin{align*}
&T = 3, \quad \mu(t) \equiv 0.15, \quad r(t) \equiv 0.08, \quad \sigma(t) \equiv 0.2,\\
&x_0 = 1, \quad S(0) = 10, \quad \lambda = 5, \quad \delta = 10,\\
&\kappa = 0.1, \quad \kappa_r = 0.15, \quad \theta = 2.
\end{align*}
\end{experiment}
\begin{figure}[htp]
	\centering
	\includegraphics[]{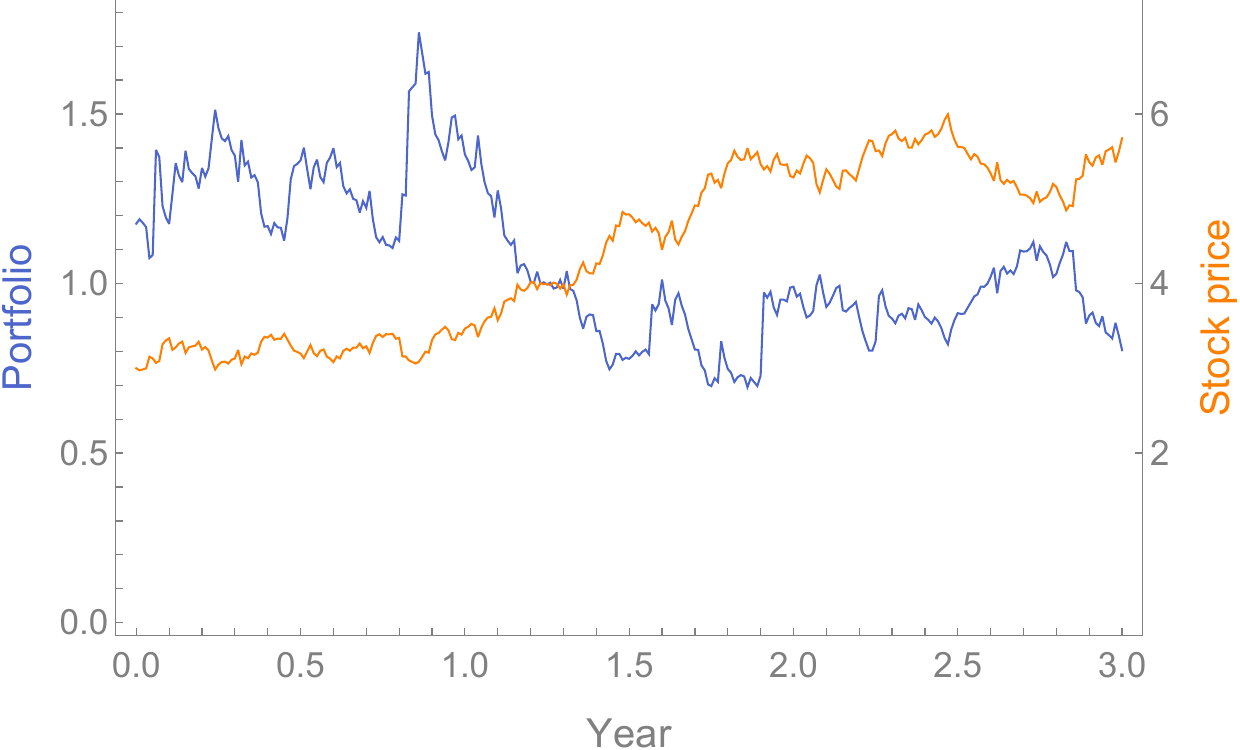}
	\caption{the trajectories of the optimal portfolio and the stock price}
	\label{mmvPC.fig.PortfolioStockPrice}
\end{figure}
\par In order to give an intuitive illustration, we simulate the trajectories of the optimal portfolio and the stock price in Figure \ref{mmvPC.fig.PortfolioStockPrice}, and the trajectories of the optimal retention level and the claims before $T = 3$ in Figure \ref{mmvPC.fig.RetentionLevelClaims}. It is easy to see that the trajectory of the optimal portfolio is opposite to that of the stock price, and the trajectory of the optimal retention level rises sharply once a claim arrives. It shows the optimal strategies are `Buy low, Sell high' strategies.
\begin{figure}[htp]
	\centering
	\includegraphics[]{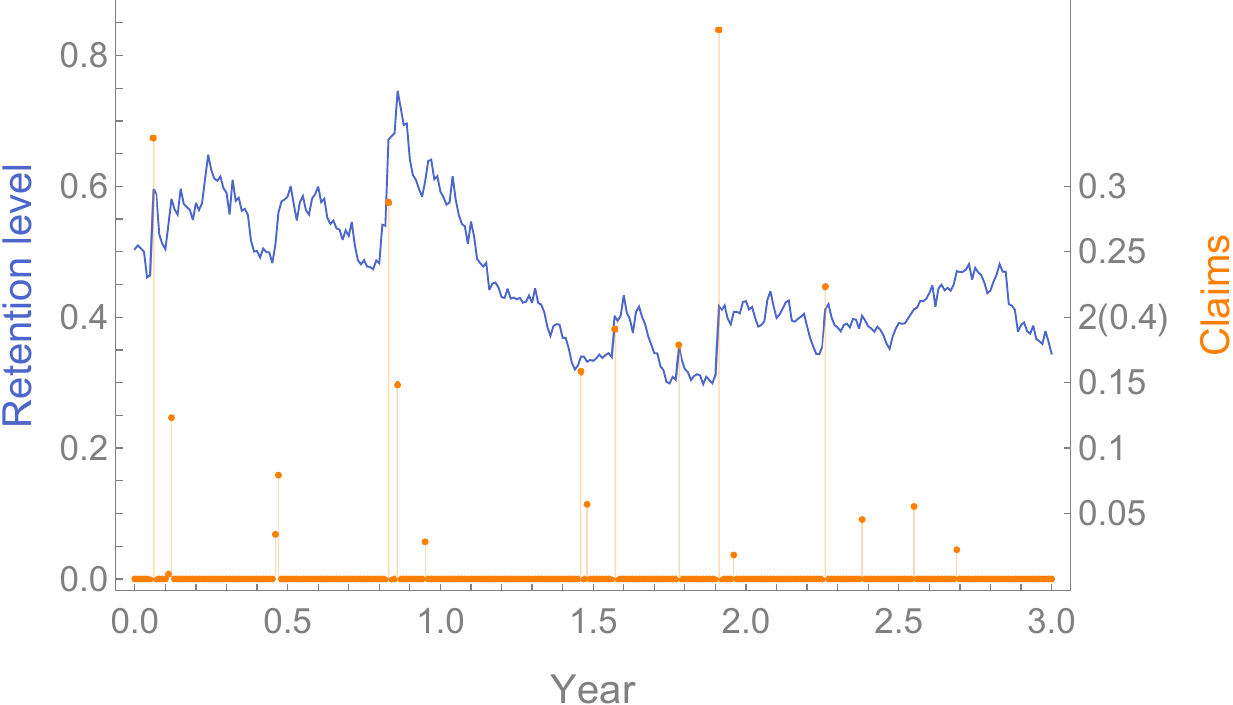}
	\caption{the trajectories of the optimal retention level and the claims}
	\label{mmvPC.fig.RetentionLevelClaims}
\end{figure}

\section{Conclusion}
\par {\color{magenta}This paper studies the optimal monotone mean-variance problem for an insurance company in the framework of the Cram{\'e}r-Lundberg risk model. The investment and the reinsurance strategies are taken into consideration. The work of this paper extends the results of \cite{li2021optimal} by incorporating jumps into the surplus process of the insurance company. To investigate the MMV problem, we first give an auxiliary two-player zero-sum SDG. A generalized exponential martingale representation theorem for nonnegative martingale is proved, which is helpful for tackling with the absolutely continuity of the disturbed probability measure $Q$. The equilibrium of the SDG is found by applying the dynamic programming approach and solving the corresponding HJBI equation.} The optimal solution of the original problem follows immediately from the equilibrium. It is proved that the optimal disturbed probability measure $Q^*$ is indeed equivalent to the real measure $P$. In the second part of the paper, we provides the efficient frontier for the MMV problem. Two numerical examples are given. To the best of our knowledge, it is the first time to study the MMV optimization problem for the model involving jump process. Possible future researches include the MMV optimization problem for more sophisticated models such as Cox process or spectral negative L{\'e}vy process. It is worth noting that, by the work of \cite{strub2020note}, the optimal strategies for the MMV and the classical MV coincide in the framework of continuous semimartingale. But it is still an open problem whether the optimal strategies always coincide for a large class of jump processes. In any case, the optimal strategy for the MMV problem makes sense economically.

\begin{appendices}

\section{From Problem {(G)} to Problem {(P$_{sxy}$)}} \label{mmvPC.section.PGtoPsxy}

\par For any market's strategy $Q \in \Delta^2(P)$, let
\begin{align}
\label{mmvPC.def.y} &Y(t) : = \frac{dQ}{dP}\big|_{\mathcal{F}_t}, \\
\nonumber &Y := Y(T) \equiv \frac{dQ}{dP}.
\end{align}
Therefore, by virtue of properties of the Radon-Nikodym derivative, we have that $Y(t) = \mathbb{E}^P[Y|\mathcal{F}_t]$. In other words, $Y(t)$  is a nonnegative martingale. In view of the fact that $\mathbb{F}$ is right-continuous, $Y(t)$ has a RCLL version which, for simplicity, we also denote it as $Y(t)$. Moreover, $Y(0) = \mathbb{E}^P[Y|\mathcal{F}_0] = 1$.

\begin{lemma} \label{mmvPC.lemma.yQ}
	\par Suppose that $Q \in \Delta^2(P)$, then the process $Y(t)$ defined by \eqref{mmvPC.def.y} is a nonnegative square-integrable $P$-martingale adapted to $\mathcal{F}_t$ with $\mathbb{E}^PY(t) = 1$. Moreover,
	\begin{equation}
	\label{mmvPC.bayes}
	\mathbb{E}^Q[X(t)] = \mathbb{E}^P\bigg[X(t)Y(t)\bigg]
	\end{equation}
	holds for any $\mathcal{F}_t$-measurable $X(t)$. Conversely, if $Y(t)$ is a nonnegative $\mathcal{F}_t$-adapted square-integrable $P$-martingale with $\mathbb{E}^PY(t) = 1$, then, for any $t \in [0,T]$, the probability measures $Q_t$ defined via
	\begin{equation}
	\label{mmvPC.def.Q} Q_t(A) = \mathbb{E}^P\bigg[1_AY(t)\bigg], \quad A \in \mathcal{F}_t
	\end{equation}
	belongs to $\Delta^2(P)$, and $Q_t$ satisfies the following consistency condition:
	\begin{equation}
	\label{mmvPC.condition.consistency}
	Q_T(A) = Q_t(A), \quad \forall A \in \mathcal{F}_t, t \in [0,T].
	\end{equation}
	Moreover, for any $\mathcal{F}_t$-measurable $X(t)$ such that $E^{Q_T}|X(t)| < \infty$, $Q_T$ satisfies \eqref{mmvPC.bayes} .
\end{lemma}

\begin{proof}
	\par Suppose that $Q \in \Delta^2(P)$. It is obvious that $Y(t)$ defined by \eqref{mmvPC.def.y} is a nonnegative martingale. By Jensen's Inequality, we have
	\begin{align*}
	\mathbb{E}^P[Y(t)^2]
	=& \mathbb{E}^P \bigg\{\bigg[\mathbb{E}^P[Y|\mathcal{F}_t]\bigg]^2 \bigg\}\\
	\leq& \mathbb{E}^P \bigg[\mathbb{E}^P[Y^2|\mathcal{F}_t] \bigg]\\
	=& \mathbb{E}^P\bigg[(\frac{dQ}{dP})^2\bigg] < +\infty, \quad \forall t \in [0,T].
	\end{align*}
	Thus $Y(t)$ is square-integrable. By the definition of Radon-Nikodym derivative, we have
	\begin{align*}
	\mathbb{E}^Q[X(t)]
	=& \mathbb{E}^P[X(t)Y] \\
	=& \mathbb{E}^P\bigg[\mathbb{E}^P[X(t)Y|\mathcal{F}_t]\bigg] \\
	=& \mathbb{E}^P\bigg[X(t)Y(t)\bigg],
	\end{align*}
	which proves \eqref{mmvPC.def.Q}.
	\par Conversely, let $Y(t)$ be a nonnegative square-integrable martingale with unit expectation. Since $Q_t(A) = \mathbb{E}^P[1_AY(t)] = 0$ whenever $P(A) = 0$, it follows from $\mathbb{E}^PY(t) = 1$ that $Q$ defined via \eqref{mmvPC.def.Q} is an absolutely continuous probability measure. The fact that $Q \in \Delta^2(P)$ is due to the square integrability of $Y(t)$. The consistency condition is due to the martingale property of $Y(t)$. In order to prove $Q_T$ satisfies \eqref{mmvPC.bayes}, we first prove the Bayes formula 
	\begin{equation}
	\label{mmvPC.pro.bayes}
	\mathbb{E}^P\bigg[X(t)Y(t)\bigg|\mathcal{F}_s\bigg] = Y(s)E^{Q_T}\bigg[X(t)\bigg|\mathcal{F}_s\bigg], \quad 0 < s < t < T.
	\end{equation}
	By using the definition of conditional expectation, we have for any $A \in \mathcal{F}_s$,
	\begin{align*}
	E^{Q_T}\bigg[1_A\frac{\mathbb{E}^P[X(t)Y(t)|\mathcal{F}_s]}{Y(s)}\bigg] =& \mathbb{E}^P\bigg[1_AY(T)\frac{\mathbb{E}^P[X(t)Y(t)|\mathcal{F}_s]}{Y(s)}\bigg] \\
	=& \mathbb{E}^P\bigg[1_A\mathbb{E}^P[Y(T)|\mathcal{F}_s]\frac{\mathbb{E}^P[X(t)Y(t)|\mathcal{F}_s]}{Y(s)}\bigg] \\
	=& \mathbb{E}^P\bigg[1_A\mathbb{E}^P[X(t)Y(t)|\mathcal{F}_s]\bigg] \\
	=& \mathbb{E}^P\bigg[1_AX(t)Y(t)\bigg] \\
	=& \mathbb{E}^P\bigg[1_AX(t)Y(T)\bigg] \\
	=& E^{Q_T}\bigg[1_AX(t)\bigg].
	\end{align*}
	The second equality and the sixth equality are due to the definition of conditional expectation. Letting $s = 0$ proves the desired result.
\end{proof}

\par Lemma \ref{mmvPC.lemma.yQ} establishes a relationship between the mean one nonnegative square-integrable martingale and the absolutely continuous probability with square-integrable density with respect to $P$. Let $\mathcal{Y}^2(P)$ be the set of all nonnegative $\mathcal{F}_t$-adapted square-integrable $P$-martingale with $\mathbb{E}^PY(t) = 1$. Thus, $Q \in \Delta^2(P)$ if and only if  $\frac{dQ}{dP}\big|_{\mathcal{F}_t} \in \mathcal{Y}^2(P)$. The MMV objective \eqref{mmvPC.obj.1} can be formulated as
\begin{align}
\label{mmvPC.I} I_{\theta}(a):
&= \min_{Y  \in \mathcal{Y}^2(P)} \left \{ \mathbb{E}^P\bigg[X^a(T)Y(T)\bigg]+\frac{1}{2 \theta}\mathbb{E}^P\bigg[Y(T)^2-1\bigg] \right \},
\end{align}
where an equivalent problem is to maximize
\begin{equation*}
\widetilde{I}_{\theta}(a): = \min_{Y  \in \mathcal{Y}^2(P)} \left \{ \mathbb{E}^P\bigg[X^a(T)Y(T)+\frac{1}{2 \theta}Y(T)^2\bigg] \right \}.
\end{equation*}
Therefore, Problem {\bf (G)} can be changed equivalently into the following problem:

\begin{problem}{\bf (D)} Let
	\begin{equation*}
	\label{mmvPC.obj.2} J_{\theta}(a , Y ): =  \mathbb{E}^P\bigg[X^a(T)Y(T)+\frac{1}{2 \theta}Y(T)^2\bigg].
	\end{equation*}
	The player one wants to maximize $J_{\theta}(a ,Y )$ with its strategy $a$ over $\mathcal{A}[0,T]$ and the player two wants to maximize $-J_{\theta}(a ,Y )$ with its strategy $Y $ over $\mathcal{Y}^2(P)$.
\end{problem}

\begin{proposition} \label{mmvPC.proposition.2}
	\par If $(a^*,Y^*)$ is a Nash equilibrium of Problem {\bf (D)} and $\mathbb{E}^{Q^*}|X^{a^*}(t)| < \infty$ where $Q^* := Q_T$ is defined by \eqref{mmvPC.def.Q} with $Y  = Y^* $, then $(a^* ,Q^*)$ is a Nash equilibrium of Problem {\bf (G)}. Conversely, if $(a^* ,Q^*)$ is a Nash equilibrium of Problem {\bf (G)}, and let $\{Y^*(t):t \in [0,T]\}$ be the process defined by \eqref{mmvPC.def.y} where $Q = Q^*$, then $(a^* ,Y^* )$ is a Nash equilibrium of Problem {\bf (D)}.
\end{proposition}
\begin{proof}
	\par We first consider the case that $(a^* ,Y^* )$ is a Nash equilibrium of Problem {\bf (D)} and $\mathbb{E}^{Q^*}|X^{a^*}(t)| < \infty$, where $Q^* := Q_T$ is defined by \eqref{mmvPC.def.Q} with $Y  = Y^* $. For any $Q \in \Delta^2(P)$, let $Y(t) : = \frac{dQ}{dP}\big|_{\mathcal{F}_t}$. It follows from Lemma \ref{mmvPC.lemma.yQ} that $Y  \in \mathcal{Y}^2(P)$. For any $a  \in \mathcal{A}[0,T]$, we have
	\begin{equation*}
	J_{\theta}(a ,Y^* ) \leq J_{\theta}(a^* ,Y^* ) \leq J_{\theta}(a^* ,Y ),
	\end{equation*}
	which, by Lemma \ref{mmvPC.lemma.yQ} and the definition of Problem {\bf (G)} and Problem {\bf (D)}, gives
	\begin{equation*}
	J_{\theta}(a ,Q^*) \leq J_{\theta}(a^* ,Q^*) \leq J_{\theta}(a^* ,Q).
	\end{equation*}
	\par For the inequality of the opposite direction, the proof is similar.
\end{proof}

\par So far, the strategy processes are $a$ and $Y $. However, $Y $ is not a tractable strategy in the theory of dynamic programming. Here we characterize it as a state process controlled by two control processes $p$ and $q$. The following theorem is a generalization of the exponential martingale representation theorem for the nonnegative martingale. 

\begin{theorem} \label{mmvPC.theorem.Qcharacterization}
	\par There exists an $\mathcal{F}_t$-adapted process $p(t)$ and an $\mathcal{F}_t$-predictable random field $q(t,z)$, such that $Y$ admits the representation
	\begin{equation}
	\label{mmvPC.sde.y.representation} Y(t) = 1 + \int_0^t Y(s-) d L(s), \quad \forall t \in [0,T],
	\end{equation}
	where $L(t) := \int_0^t p(s)dW(s) + \int_0^t \int_{K} q(s,z) \widetilde{N}(ds,dz)$ with $\Delta L(t) \geq -1$, $t \in [0,T]$. Besides, $p(t)$ and $q(t,z)$ satisfy
	\begin{equation}
	\label{mmvPC.condition.integrable.property}
	E \bigg[ \int_0^{T \wedge \zeta_n} p(t)^2 dt + \int_0^{T \wedge \zeta_n} \int_{K} (1 - \sqrt{q(t,z)+1})^2 v(dz) dt \bigg] < \infty,
	\end{equation}
	for any integer $n \geq 1$, where 
	\begin{equation*}
	\zeta_n := \inf\left\{ t \geq 0; Y(t) \leq \frac{1}{n}\right\}.
	\end{equation*}
\end{theorem}

\begin{proof} 
	\par See Appendix B.
\end{proof}

\begin{remark}
	\par Let $\mathcal{P}$ be the predictable $\sigma$-algebra. Let $\widetilde{\mathcal{P}} = \mathcal{P} \otimes \mathcal{K}$, where $\mathcal{K}$ is the Borel $\sigma$-algebra of $K$. Let $M_N^P$ be the Dol{\'e}ans measure associated with the random measure $N$ under the probability measure $P$. Then $p(t)$ and $q(t,z)$ have the following representations
	\begin{equation}
	\label{mmvPC.b.representation} p(t) =  Y^{\oplus}(t) \frac{d \langle Y,W \rangle(t)}{d t}, \quad q(t,z) = Y^{\oplus}(t-) M_N^P(Y|\widetilde{\mathcal{P}})(t,z) - 1,
	\end{equation}
	where $Y^{\oplus}(t) = Y(t)^{-1}1_{\{Y(t)>0\}}$. Given a process $Y \in \mathcal{Y}^2$, it follows from Lemma 7 (2) of \cite{kabanov1979absolute} that SDE \eqref{mmvPC.sde.y.representation} with $p(t)$ and $q(t,z)$ defined by \eqref{mmvPC.b.representation} has a unique solution which equals $Y$ (indistinguishable).
\end{remark}

\par Denote $b(t,z) := (p(t), q(t,z))$ and $ \mathcal{B}[0,T] = \{ b(\cdot,\cdot) |$ for any $z \in K$, $ q(\cdot,z): [0,T] \to \mathbb{R}$ and $p : [0,T] \to \mathbb{R}$ are $\mathcal{F}_t$-predictable and $\mathcal{F}_t$-adapted, respectively, satisfying $\int_{K} q(t,z) N(\{t\},dz) \geq -1, \ \forall t \in [0,T]$, such that SDE \eqref{mmvPC.sde.y.representation} has a unique solution belonging to $\mathcal{Y}^2$ $\} $. The following corollary provides a theoretical basis for changing the control process from $Y $ to $p$ and $q$.

\begin{corollary} \label{mmvPC.corollary.1}
	\par Suppose that $Y  \in \mathcal{Y}^2$, then the process $b(t,z)$ defined by \eqref{mmvPC.b.representation} belongs to $\mathcal{B}[0,T]$ and $Y $ is the unique solution to  \eqref{mmvPC.sde.y.representation}. Conversely, if $b(t,z) \in \mathcal{B}[0,T]$, then the process $\{ Y(t): t \in [0,T] \}$ defined by \eqref{mmvPC.sde.y.representation} belongs to $\mathcal{Y}^2$.
\end{corollary}
\begin{proof}
	\par This corollary follows immediately from Theorem \ref{mmvPC.theorem.Qcharacterization}.
\end{proof}

\par The following proposition provides the equivalence of the solutions of Problem {\bf (D)} and Problem {\bf (P$_{sxy}$)}.

\begin{proposition} \label{mmvPC.proposition.3}
	\par If $(a^*(\cdot;0,x_0,1),b^*(\cdot,\cdot;0,x_0,1))$ is a Nash equilibrium of Problem {\bf (P$_{0x_01}$)}, then $(a^*(\cdot;0,x_0,1),Y^*(\cdot;0,x_0,1))$ is a Nash equilibrium of Problem {\bf (D)} where $Y^*(\cdot;0,x_0,1)$ is the solution to \eqref{mmvPC.dynamic.y} with $b(\cdot,\cdot) = b^*(\cdot,\cdot;0,x_0,1)$. Conversely if $(a^*(\cdot;0,x_0,1),Y^*(\cdot;0,x_0,1))$ is a Nash equilibrium of Problem {\bf (D)}, then $(a^*(\cdot;0,x_0,1),b^*(\cdot,\cdot;0,x_0,1))$ is a Nash equilibrium of Problem {\bf (P$_{0x_01}$)} where $b^*(\cdot,\cdot;0,x_0,1)$ is defined via \eqref{mmvPC.b.representation}.
\end{proposition}

\section{Proof of Theorem \ref{mmvPC.theorem.Qcharacterization}}

\par We first give some lemmas before proving Theorem \ref{mmvPC.theorem.Qcharacterization}. 

\par Note that $Y$ may hit zero at finite time, we need to consider the behavior of $Y$ after that. Define 
\begin{equation*}
\zeta = \inf\{ t \geq 0; Y(t) \leq 0\},
\end{equation*}
with the convention that $\inf\emptyset=+\infty$. Then $Y(\zeta)=0$, $P$-a.s. on the set $\{\zeta \leq T\}$. The following lemma shows that $Y$ would stay there once it hit zero.

\begin{lemma} \label{mmvPC.lemma.Y1}
	\par On the set $\{\zeta \leq T\}$, it holds almost surely under probability $P$ that
	\begin{equation*}
	Y(\zeta+t) = 0, \quad t \geq 0.
	\end{equation*}
\end{lemma}
\begin{proof}
	Given a fixed $t \geq 0$, note that $Y$ is a nonnegative martingale, then by applying the optional stopping theorem, we have
	\begin{align*}
	0 \leq \mathbb{E}^P\left[Y(\zeta+t)1_{\{\zeta \leq T\}}\right] =& \mathbb{E}^P\left[\mathbb{E}^P\left[Y(\zeta+t)|\mathcal{F}_{\zeta}\right]1_{\{\zeta \leq T\}}\right] \\
	\leq& \mathbb{E}^P\left[Y(\zeta)1_{\{\zeta \leq T\}}\right] = 0.
	\end{align*}
	Since $Y(\zeta+t)1_{\{\zeta \leq T\}}$ is nonnegative $P$-a.s., we obtain $Y(\zeta+t)1_{\{\zeta \leq T\}} = 0$ $P$-a.s.. Let $N_t$ denote the set such that $Y(\zeta+t)1_{\{\zeta \leq T\}} \neq 0$, then $P(N_t)=0$. Let $N := \cup_{t \in \mathbb{Q}_+} N_t$, where $\mathbb{Q}_+$ is the set of all positive rational number. Then, for all $\omega \in \Omega \setminus N$, $Y(\zeta+t)1_{\{\zeta \leq T\}} = 0$ holds for all $t \in Q_+$. By the right-continuity of $Y$, we have that $Y(\zeta+t)1_{\{\zeta \leq T\}} = 0$ holds for all $t \in R_+$ $P$-a.s..
\end{proof}

\begin{proofth3.1}
Let $Y^{\zeta}  := Y(\cdot \wedge \zeta)$. By (2), (3) of Theorem 13 and (1) of Lemma 7 of \cite{kabanov1979absolute}, $Y$ admits the representation
 \begin{equation}
	\label{mmvPC.sde.1} Y^{\zeta}(t) = 1 + \int_0^t Y^{\zeta}(s-) d L(s), \quad \forall t \in [0,T],
\end{equation}
where $L(t) := \int_0^t p(s)dW(s) + \int_0^t \int_{K} q(s,z) \widetilde{N}(ds,dz)$ with $\Delta L(t) \geq -1$, $t \in [0,T]$ satisfying
\begin{equation*}
	E \bigg[ \int_0^{T \wedge \zeta_n} p(t)^2 dt + \int_0^{T \wedge \zeta_n} \int_{K} (1 - \sqrt{q(t,z)+1})^2 v(dz) dt \bigg] < \infty,
\end{equation*}
for any $\zeta_n$, where 
\begin{equation*}
\zeta_n := \inf\left\{ t \geq 0; Y(t) \leq \frac{1}{n}\right\}.
\end{equation*}

For the case that $\{ \zeta > T\}$, we can simply remove the superscript in \eqref{mmvPC.sde.1}, which gives \eqref{mmvPC.sde.y.representation}. For the case that $\{ \zeta \leq T\}$, if $t \leq \zeta$, then $Y^{\zeta}(t) = Y(t \wedge \zeta) = Y(t)$; if $t > \zeta$, then $Y^{\zeta}(t) = Y(\zeta) = 0 = Y(t)$, by using Lemma \ref{mmvPC.lemma.Y1}. Therefore, on $\{ \zeta \leq T\}$, we have $Y^{\zeta}(t) = Y(t)$ and \eqref{mmvPC.sde.y.representation} holds.
\end{proofth3.1}

\par
\par

\end{appendices}

\noindent {\\
\bf Acknowledgment} This work was supported by the National Natural Science Foundation of China 11931018, the Tianjin Natural Science Foundation 19JCYBJC30400, the National Natural Science Foundation of China 12201104 and the Fundamental Research Funds for the Central Universities 2232021D-29.


\end{document}